\documentclass[a4paper]{article}

\usepackage[T1]{fontenc}
\usepackage{a4}
\usepackage[english]{babel}
\usepackage[dvipsnames,usenames]{color}

\usepackage{bbm}

\usepackage{verbatim,graphics,indentfirst} 
 \usepackage{url}
\usepackage{stmaryrd}
\usepackage[nosumlimits]{amsmath}
 \usepackage{amssymb}
\usepackage{amsthm}

\usepackage{latexsym}
\usepackage{enumerate}
\usepackage{graphicx}
\usepackage[T1]{fontenc}
\usepackage{lmodern}

\usepackage[english]{babel}
\usepackage{caption}
\usepackage{afterpage}

\usepackage{hyperref}
\usepackage{mathrsfs}
\usepackage{setspace}
\usepackage[all]{xy}

\usepackage[table]{xcolor}

\usepackage{tikz}
\usetikzlibrary{shapes.arrows,chains}
\usetikzlibrary {positioning}
\usetikzlibrary{arrows}
\usepackage{tkz-graph}
\usepackage{float}

\usepackage{array}
\newcommand{\PreserveBackslash}[1]{\let\temp=\\#1\let\\=\temp}
\newcolumntype{C}[1]{>{\PreserveBackslash\centering}p{#1}}
\newcolumntype{R}[1]{>{\PreserveBackslash\raggedleft}p{#1}}
\newcolumntype{L}[1]{>{\PreserveBackslash\raggedright}p{#1}}
 
\theoremstyle{definition}
\newtheorem{definition}{\vspace{1mm}Definition}

\theoremstyle{plain}
\newtheorem{lemma}[definition]{\vspace{1mm}Lemma}
\newtheorem{theorem}[definition]{\vspace{1mm}Theorem}
\newtheorem{corollary}[definition]{\vspace{1mm}Corollary}
\newtheorem{proposition}[definition]{\vspace{1mm}Proposition}

\newtheorem{example}[definition]{\vspace{1mm}Example}

\newcommand{\conn}{{\copyright}}

\newcommand{\nats}                  {{\mathbb N}}

\newcommand{\Mt}{\mathbb{M}}
\newcommand{\Ut}{\mathbb{U}}
\newcommand{\Ct}{\mathbb{C}}
\newcommand{\KMt}{\mathbb{K}}

\newcommand{\der}    						  {\vartriangleright}

\newcommand{\tuple}[1]                         {{\langle #1\rangle}}

\DeclareMathOperator*{\sub}{\mathsf{sub}}

\newcommand{\qi}{q_{\mathsf{init}}}
\newcommand{\incr}[1]{\textsf{inc}(#1)}
\newcommand{\test}[1]{\textsf{dec}(#1)}

\newcommand{\step}{\mathsf{step}}

\newcommand{\suc}{\mathsf{succ}}
\newcommand{\enc}{\mathsf{enc}}

\newcommand{\n}{\textrm{r}}

\newcommand{\seq}{\mathsf{seq}} 
\newcommand{\mC}{\mathcal{C}}
\newcommand{\Conf}{\mathbf{Conf}} 
\newcommand{\HConf}{\mathbf{HConf}} 

\newcommand{\Rm}{\textbf{Num}}
\newcommand{\zero}{\mathsf{zero}}

\newcommand{\It}{\mathbb{I}}

\newcommand{\limp}          {\Rightarrow}

\DeclareMathOperator*{\inc}{\mathsf{inc}}

\DeclareMathOperator*{\nxt}{\mathsf{nxt}}

\newcommand{\Var}{\mathsf{var}}
\newcommand{\Sub}{\mathsf{sub}}
\newcommand{\sder}{\vdash}
\newcommand{\Thm}{\mathsf{Thm}}

\newcommand{\Val}{\textrm{Val}}
\newcommand{\BVal}{\textrm{BVal}}

 \newcommand{\bvals}{{\mathsf B}}
    \newcommand{\bval}{{\mathsf b}}

\date{}

\title{Equivalence of finite non-deterministic logical matrices is undecidable\thanks{This work was supported by FCT - Funda\c c\~ao para a Ci\^encia e a Tecnologia, I.P. by project reference 
UIDB/50008/2020, and DOI identifier \url{https://doi.org/10.54499/UIDB/50008/2020} .The second author acknowledges the grant PD/BD/135513/2018 by FCT, under the LisMath PhD programme.
}} 

\author{Carlos Caleiro, Pedro Filipe, S\'ergio Marcelino\\
{SQIG - Instituto de Telecomunica\c c\~oes}\\
{Dep. Matem\'atica - Instituto Superior T\'ecnico}\\
{Universidade de Lisboa, Portugal}}

\begin{document}

\maketitle

\begin{abstract}
The notion of a non-deterministic logical matrix (where connectives are interpreted as multi-functions) extends the traditional semantics for propositional logics based on
logical matrices (where connectives are interpreted as functions). This extension allows for finitely characterizing a much wider class of logics, and has proven decisive in a myriad
of recent compositionality results. In this paper we show that the added expressivity brought by non-determinism also has its drawbacks, and in particular that the problem of determining whether two given finite non-deterministic matrices are equivalent, in the sense that they induce the same logic, becomes undecidable. We also discuss some workable sufficent conditions and particular cases, namely regarding rexpansion homomorphisms and bridges to calculi.
\end{abstract}

\section{Introduction}
Logical matrices are arguably the most widespread semantic structures used to characterize propositional logics~\cite{Wojcicki88,AlgLogBook}. After {\L}ukasiewicz, 
a logical matrix consists of an underlying algebra, functionally interpreting logical connectives over a set of truth-values, together with a designated set of truth-values. 
The logical models (valuations) are obtained by considering homomorphisms from the free-algebra in the matrix similarity type into the underlying algebra of the matrix, and formulas that hold in 
the model are the ones that take designated values. 

However, in recent years, it has become clear that there are advantages in departing from semantics based on logical matrices, by adopting a non-deterministic generalization 
of the standard notion where logical connectives are interpreted by multi-functions instead of functions. Valuations are still defined homomorphically from the free-algebra, but now the valuation of a formula with a certain head connective can be picked non-deterministically from the set of possible values permitted by the multi-function interpreting the connective, instead of being completely determined by the values assigned to its immediate subformulas. The systematic study of non-deterministic logical matrices (Nmatrices) and their applications was initiated in the beginning of this century by Avron and his collaborators~\cite{avr:zam:surveyNDS,Avron:Lev:NDMVS}.  Logical semantics based on Nmatrices are very malleable, allowing not only for finite characterizations of 
logics that do not admit finite-valued semantics based on logical matrices, but also permitting general recipes for various practical problems in logic, an in particular compositionality~\cite{wollic17,newfibring,AvronBook}. Still, as noticed by Avron himself, and quoting Zohar~\cite{ZoharPhD} in his PhD thesis:
{\quote{\it ``An interesting direction for further research is to find a necessary and sufficient criterion for two Nmatrices to induce the same logic ...''}}\\

The pertinence of this question is better understood in contrast to the deterministic case. Indeed, the problem of determining whether two given finite logical matrices are equivalent, in the sense that they induce the same logic, is decidable. Although such a decision procedure cannot be immediately found in the literature, to the best of our knowledge, the result is somewhat folklore and a relatively straightforward corollary of well known properties of logical matrices. However, these results do not extend to the non-deterministic case, as argued also in~\cite{CZE}. \\

The present paper is a natural extension of preliminary results presented in~\cite{FilipeJLC,FilipeNCL}, where we prove the undecidability of two relevant computational problems associated with finite Nmatrices: given a finite Nmatrix, the problem of determining whether its logic has any theorem whatsoever, and the problem of determining whether the Nmatrix is \emph{monadic}\footnote{Monadicity is an expressiveness requirement that plays an essential role in the synthesis of analytic calculi for the logic (see~\cite{SYNTH}).}. Expectedly, these undecidability results were obtained using reductions from other known undecidable problems, namely the universality problem for term-dag automata~\cite{AutomataOnDAGRepresentationsOfFiniteTrees}, and the halting problem for Minsky's counter machines~\cite{minsky}. 

Our main result in the present paper is precisely the undecidability of the equivalence problem in the presence of non-determinism, and its proof relies on a reduction from the above mentioned theorem existence problem. The reduction uses an interesting and surprisingly simple trick allowed by the non-deterministic environment, which bears similarities with results about infectious semantics such as~\cite{ISMVL}.
For the sake of self-containment, but also due to the key role it plays and the deep illustration it provides about the power of non-determinism, we shall also carefully revisit the undecidability result for the theorem existence problem and its reduction from the halting problem for counter machines.

Our result helps explain some of the difficulties posed by studying logics characterized by finite Nmatrices. In concrete cases, showing whether two such logics coincide, or not, may be quite challenging. Still, we shall overview a few useful techniques and sufficient criteria that can be used in practice to deal with such problems.  \\

The paper is organized as follows. In Section~\ref{sec2} we recall the essential notions of Tarskian logic, logical matrices and Nmatrices, introduce the above mentioned computational problems, and recall the decidability of the equivalence problem for finite logical matrices. In Section~\ref{sec3}, we revisit Minsky's counter machines and the undecidability of the theorem existence problem for Nmatrices, obtained via a computable reduction from the halting problem for counter machines. Section~\ref{sec4} establishes our main results, a useful construction on Nmatrices permitted by non-determinism that allows isolating theoremhood from a given Nmatrix, and ultimately the undecidability of the equivalence problem for Nmatrices via a dual reduction from the theorem existence problem. Finally, Section~\ref{sec5} is devoted to overviewing a number of sufficient conditions for equivalence, and an analysis of some concrete examples. We conclude, in Section~\ref{sec6:conc}, with a discussion of the results obtained and the challenges they pose to a systematic algebraic-like study of non-deterministic logical matrices.

\section{Basic notions and results}\label{sec2}

In this section we introduce the relevant notions regarding (propositional-based) logics, as well as logical matrices and Nmatrices, the computational problems we will look into, and their decidability with respect to finite matrices.

\subsection{Logics}

A \emph{signature} $\Sigma$ is a family of sets of \emph{connectives} indexed by their arity, $\Sigma = \{\Sigma^{(k)} : k \in \nats_0\}$.
We will say that $\Sigma$ is a \emph{finite signature} if it comprises only finitely many connectives, that is, $\Sigma^{(k)}$ is finite for all $k \in \nats_0$, and $\{k\in\nats_0:\Sigma^{(k)}\neq\emptyset\}$ is also finite. 
Let $P$ be a set of \emph{propositional variables}, we denote by $L_{\Sigma}(P)$ the \emph{set of formulas} based on $P$ built from the connectives in $\Sigma$. Unless specified otherwise, we assume $P = \{p_i : i \in \nats\}$
and, for every $n \in \mathbb{N}$, denote by $P_n$ the set $P = \{p_i :  i \leq n\}$.
In general, we use roman letters ($A,B, \dots$) to represent formulas and capital greek letters ($\Gamma, \Delta, \dots$) to represent sets of formulas.
For every $A \in L_{\Sigma}(P)$, we denote by $\Sub(A)$ and $\Var(A)$, respectively, the \emph{set of subformulas} and the \emph{set of variables} of $A$. As usual, we say that $A$ is a \emph{closed formula} if $\Var(A)=\emptyset$.
A \emph{substitution} is a function $\sigma : P \to L_{\Sigma}(P)$, uniquely extendable to an endomorphism ${\sigma} : L_{\Sigma}(P) \to L_{\Sigma}(P)$ (for simplicity, we use the same denomination for both).
For every $\Gamma\cup\{A\} \subseteq L_{\Sigma}(P)$ and every substitution $\sigma$, we use $A^\sigma$ to denote the formula $\sigma(A)$, and $\Gamma^{\sigma}$ to denote the set $\{A^{\sigma} : A \in \Gamma\}$.\smallskip

A \emph{(Tarskian) consequence relation}, also called a \emph{logic}, is a pair $\tuple{\Sigma,\sder}$ where $\Sigma$ is a signature and $\sder {\subseteq}\, {\mathbf{2}^{L_{\Sigma}(P)} \times L_{\Sigma}(P)}$ is a relation such that the following properties hold,
for all $\Gamma \cup \Delta \cup \{A\} \subseteq L_{\Sigma}(P)$:
\begin{itemize}
    \item if $A \in \Gamma$ then $\Gamma \sder A$ \emph{(reflexivity)},
    \item if $\Gamma \sder A$ and $\Gamma \subseteq \Delta$ then $\Delta \sder A$ \emph{(monotonicity)},
    \item if $\Gamma \sder A$ and $\Delta \sder B$, for all $B \in \Gamma$, then $\Delta \sder A$ \emph{(transitivity)},
    \item if $\Gamma \sder A$ then $\Gamma^{\sigma} \sder A^{\sigma}$, for every substitution $\sigma$ \emph{(substitution-invariance)}.
\end{itemize}

Additionally, a consequence relation $\tuple{\Sigma,\sder}$ is said to be \emph{compact} if the following property also holds:
\begin{itemize}
    \item if $\Gamma \sder A$ then there is finite $\Delta\subseteq\Gamma$ such that $\Delta \sder A$ \emph{(finitariness)}.
\end{itemize}

A formula $A\in L_{\Sigma}(P)$ is said to be a \emph{theorem} of $\tuple{\Sigma,\sder}$ if $\emptyset \sder A$, and we denote by $\Thm(\Sigma,\sder)$ the set of all its theorems.\\

A consequence relation $\tuple{\Sigma,\sder}$ is said to be \emph{$n$-determined} for $n \in \mathbb{N}$, if for all $\Gamma \cup \{A\} \subseteq L_{\Sigma}(P)$, whenever
$\Gamma \not\sder A$ then there exists a substitution $\sigma : P \to  P_n$
 such that $\Gamma^{\sigma} \not\sder A^{\sigma}$.
 We say that $(\Sigma,\sder)$ is \emph{finitely-determined} when it is $n$-determined for some $n \in \mathbb{N}$. 
 
Further, $\tuple{\Sigma,\sder}$ is said to be \emph{locally tabular} if, for every $n\in\nats$, $L_\Sigma(P_n)$ is partitioned into finitely many equivalence classes by the Frege interderivability relation, that is, there exists a finite set $\Delta_n\subseteq L_\Sigma(P)$ such that for each $A\in L_\Sigma(P_n)$ there exists $A^*\in\Delta_n$ such that $A\dashv\sder A^*$, i.e., $\{A\}\sder A^*$ and $\{A^*\}\sder A$.

\subsection{Matrices and Nmatrices}

A \emph{non-deterministic matrix (Nmatrix)} over signature $\Sigma$ is a triple $\Mt = \tuple{V,D,\cdot_{\Mt}}$, where
$V$ is a non-empty set of \emph{truth-values}, $D\subseteq V$ is a set of \emph{designated} truth-values and, for every $k \in \nats_0$ and
$\conn \in \Sigma^{(k)}$, $\conn_{\Mt} : V^k \to \mathbf{2}^V\setminus\{\emptyset\}$ is the \emph{interpretation} of $\conn$ in $\Mt$.
A truth-value is said to be \emph{undesignated} if it belongs to $V \setminus D$. We refer to the pair $\tuple{V,\cdot_{\Mt}}$ as the
\emph{underlying multialgebra} of $\Mt$, and usually dub $\Mt$ a $\Sigma$-Nmatrix. We will say that $\Mt$ is \emph{finite-valued} provided that $V$ is a finite set, and further say that $\Mt$ is \emph{finite} if $\Sigma$ is also a finite signature.

$\Mt$ is said to be 
\emph{deterministic, or simply a $\Sigma$-matrix} if, for every $k \in \nats_0$, $\conn \in \Sigma^{(k)}$ and $x_1,\dots,x_k \in V$, we have that $\conn_{\Mt}(x_1,\dots,x_k)$ is a singleton. In that case, the underlying multialgebra $\tuple{V,\cdot_{\Mt}}$ can indeed be understood as an algebra with each $\conn_\Mt$ as a function of type $V^k\to V$.

A \emph{valuation} of $\Mt$ is a function $v : L_{\Sigma}(P) \to V$ such that, for every $k \in \nats_0$, $\conn \in \Sigma^{(k)}$ and $A_1,\dots,A_k \in L_\Sigma(P)$ it is the case that 
\begin{equation}\label{valcond}
    v(\conn(A_1,\dots,A_k)) \in \conn_{\Mt}(v(A_1),\dots,v(A_k)).
\end{equation}
The \emph{set of all valuations of $\Mt$} is denoted by $\Val(\Mt)$.
As is well known, if $\Mt$ is a matrix then every function $v : Q \to V$, where $Q\subseteq P$, can be extended to a valuation of $\Mt$, and the extension is unique for formulas in $L_\Sigma(Q)$.
When building a valuation for some Nmatrix $\Mt$, however, one needs to choose values for complex formulas as well.
Still, in general, a useful form of locality for valuations in Nmatrices still holds. Namely,  
any function $v : \Gamma \to V$, where $\Sub(\Gamma)\subseteq\Gamma \subseteq L_{\Sigma}(P)$, can be extended to a valuation of $\Mt$, provided that
condition~(\ref{valcond}) holds for every formula in $\Gamma\setminus P$, in which case $v$ is dubbed a \emph{prevaluation}. \\

Every $\Sigma$-Nmatrix $\Mt = \tuple{V,D,\cdot_{\Mt}}$ defines a logic $\tuple{\Sigma,\sder_{\Mt}}$ in the following standard way: for every $\Gamma \cup \{A\} \subseteq L_{\Sigma}(P)$,
let $\Gamma \sder_{\Mt} A$ if and only if, for every $v \in \Val(\Mt)$, $v(A) \in D$ whenever $v(\Gamma) \subseteq D$. When $\Mt$ is finite-valued it is well known that $(\Sigma,\sder_{\Mt})$ is always compact~\cite[Theorem 3.15]{Avron:Lev:NDMVS}.

In what follows, we will say that $\Sigma$-Nmatrices $\Mt_1,\Mt_2$ are \emph{equivalent} provided that $\sder_{\Mt_1}{=}\sder_{\Mt_2}$.\\

Fixed a $\Sigma$-Nmatrix $\Mt = \tuple{V,D,\cdot_{\Mt}}$, a formula $A\in L_\Sigma(P_n)$ induces a multi-function $[A]_{\Mt} : V^n \to \mathbf{2}^V\setminus\{\emptyset\}$, where for $\vec{x}=(x_1,\dots,x_n)\in V^n$ one has:
\begin{equation*}
    [A]_{\Mt}(\vec{x}) = \{v(A) : v \in \Val(\Mt) \text{ such that } v(p_i) = x_i \text{, for all } i=1,\dots,n\}.
\end{equation*}
The multi-function $[A]_{\Mt}$ is said to be \emph{expressed} by the formula $A$ in $\Mt$, and thus \emph{expressible} in $\Mt$.
Note that, whenever $\Mt$ is deterministic, the set $[A]_{\Mt}(\vec{x})$ is always a singleton and, therefore, we can understand $[A]_\Mt$ as a function of type $V^n\to V$.
Indeed, if we do that, given $\conn(A_1,\dots,A_k)\in L_\Sigma(P_n)$, we have that $v(\conn(A_1,\dots,A_k))=\conn_\Mt(v(A_1),\dots,v(A_k))$ for any $v\in\Val(\Mt)$, and therefore also

\begin{equation}\label{valmatr}
    [\conn(A_1,\dots,A_k)]_\Mt(\vec{x})=\conn_\Mt([A_1]_\Mt(\vec{x}),\dots,[A_k]_\Mt(\vec{x})).
\end{equation}

These facts explain the following property, that we recall from~\cite{finval}.

\begin{proposition}\label{finval}
Let $n\in\nats$. If $\Mt = \tuple{V,D,\cdot_{\Mt}}$ is a finite-valued $\Sigma$-matrix then $\tuple{\Sigma,\sder_{\Mt}}$ is locally tabular and $n$-determined if $n\geq |V|$.
\end{proposition}
\proof{We start with local tabularity. Since there are finitely many functions of type $V^n\to V$ (exactly $|V|^{|V|^n}$) it follows that the set $\{[A]_{\Mt}:A\in L_\Sigma(P_n)\}$ is finite. Further, 
assuming that for formulas $A,B\in L_\Sigma(P_n)$ one has $[A]_{\Mt}=[B]_{\Mt}$ then it follows that $A\dashv\sder_\Mt B$. Namely, given $v\in\Val(\Mt)$, if $v(p_1)=x_1,\dots,v(p_n)=x_n$ one has $v(A)=[A]_\Mt(x_1,\dots,x_n)=[B]_\Mt(x_1,\dots,x_n)=v(B)$ and thus $v(A)\in D$ if and only if $v(B)\in D$. Local tabularity follows by letting $\Delta_n\subseteq L_\Sigma(P_n)$ contain exactly one representative formula for each of the finitely many $n$-place functions expressible in $\Mt$.\smallskip

Concerning $n$-determinedness, now, let $V=\{x_1,\dots,x_m\}$ with $n\geq m$, and assume that $\Gamma\not\sder_{\Mt}A$ for some $\Gamma\cup\{A\}\subseteq L_\Sigma$. By definition, we know that there exists $v\in\Val(\Mt)$ such that $v(\Gamma)\subseteq D$ and $v(A)\notin D$. Consider the substitution $\sigma:P\to P_n$ defined, for each $p\in P$, by $\sigma(p)=p_i$ if and only if $v(p)=x_i$. We claim that $\Gamma^\sigma\not\sder_{\Mt}A^\sigma$. To see this, just let $v'\in\Val(M)$ be any valuation such that $v'(p_i)=x_i$ for $i=1,\dots,m$. A straightforward induction on $B\in L_\Sigma(P)$ shows that $v'(B^\sigma)=v(B)$. 
Thus, we conclude that $v'(\Gamma^\sigma)=v(\Gamma)\subseteq D$, $v'(A^\sigma)=v(A)\notin D$, and it follows that $\Gamma^\sigma\not\sder_{\Mt}A^\sigma$.\qed}\\

A compositionality property such as (\ref{valmatr}) is not necessarily true for Nmatrices. One still has that $[\conn(A_1,\dots,A_k)]_\Mt(\vec{x})\subseteq\bigcup_{y_1\in [A_1]_\Mt(\vec{x}),\dots,y_k\in[A_k]_\Mt(\vec{x})}\conn_\Mt(y_1,\dots,y_k)$ but equality does not hold in general (just suppose that $A_i=A_j=A$ with $i\neq j$ and that $[A]_\Mt(\vec{x})$ has more than one element, and note that necessarily $v(A_i)=v(A_j)$ for every $v\in\Val(\Mt)$). 
Further, a finite-valued Nmatrix may well fail to be finitely-determined, or locally tabular, as illustrated in the examples below.

\begin{example}{(Unconstrained Nmatrices)}\label{unconstrained}\\
Given signature $\Sigma$, the \emph{unconstrained} $\Sigma$-Nmatrix is $\Ut_\Sigma=\tuple{\{0,1\},\{1\},\cdot_{\Ut_\Sigma}}$ where $\conn_{\Ut_\Sigma}(\vec{x})=\{0,1\}$ for each $k\in\nats_0$, $\conn\in\Sigma^{(k)}$ and $\vec{x}=(x_1,\dots,x_n)\in \{0,1\}^n$.

It is clear that $\Val(\Ut_\Sigma)=\{0,1\}^{L_\Sigma(P)}$ and thus that the Nmatrix defines the \emph{discrete} logic such that $\Gamma\sder_{\Ut_\Sigma}A$ if and only if $A\in\Gamma$, for $\{A\}\cup\Gamma\subseteq L_\Sigma(P)$.\smallskip

As long as $\Sigma$ is non-empty, such a discrete logic cannot be defined by a finite-valued $\Sigma$-matrix, as shown in~\cite{Avron:Lev:NDMVS}. Further, when $\Sigma$ is non-empty
it is also worth noting that $\tuple{\Sigma,\sder_{\Ut_\Sigma}}$ fails to be locally tabular. Indeed, for every $n\in\nats$ and $A\in L_\Sigma(P_n)\setminus P_n$ it is immediate that $[A]_{\Ut_\Sigma}(\vec{x})=\{0,1\}$ for all $\vec{x}$. This actually means that any two formulas with at least one connective each express the same multi-function. However, as we have already seen, if $B\in L_\Sigma(P_n)\setminus P_n$ is such that $B\neq A$ then $A\not\sder_{\Ut_\Sigma} B$ (and $B\not\sder_{\Ut_\Sigma} A$).
\hfill$\triangle$
\end{example}

\begin{example}{(A single 1-place connective)}\label{unarys}\\
Let $\Sigma$ be a signature with a single 1-place connective, i.e., $\Sigma^{(1)}=\{\flat\}$ and $\Sigma^{(k)}=\emptyset$ if $k\neq 1$. 
Consider the two-valued $\Sigma$-Nmatrix $\Ut_\Sigma$, and its possible refinements $\Mt_i=(\{0,1\},\{1\},\cdot_{\Mt_i})$ with $i=1,\dots,8$ given by:

\begin{center}
\begin{tabular}{C{5pt}|C{5mm}|C{5mm}|C{5mm}|C{5mm}|C{5mm}|C{5mm}|C{5mm}|C{5mm}|C{5mm}}
 & ${\flat_{\Ut_\Sigma}}$ & ${\flat_{\Mt_1}}$  & ${\flat_{\Mt_2}}$  & ${\flat_{\Mt_3}}$  & ${\flat_{\Mt_4}}$  & ${\flat_{\Mt_5}}$  & ${\flat_{\Mt_6}}$  & ${\flat_{\Mt_7}}$  & ${\flat_{\Mt_8}}$ \\
\hline
$0$& $ 0,1 $ & $ 0,1 $ & $ 0,1 $ & $ 0 $ & $ 0 $ & $ 0 $ & $ 1 $& $ 1 $& $ 1 $\\
$1$ &$ 0,1 $ &$ 0 $ &$ 1 $ &$ 0,1 $ &$ 0 $ &$ 1 $ &$ 0,1 $ &$ 0 $ &$ 1 $
\end{tabular}
\end{center}

These nine distinct Nmatrices define eight distinct logics. For instance, one has $\sder_{\Mt_1}{\neq}\sder_{\Ut_\Sigma}$ as can be easily noted from the fact that $\{A,\flat A\}\sder_{\Mt_1} B$ for all $A,B\in L_\Sigma(P)$. A thorough analysis of why most of these logics are distinct, and in particular to identifying which pair of these Nmatrices are equivalent, are postponed to the examples in Section~\ref{sec5}.
\hfill$\triangle$
\end{example}

{
\begin{example}{(Kearns modal semantics without possible worlds)}\label{something}\\
Let $\Sigma$ be a signature with two 1-place connectives $\neg$ and $\Box$, and two 2-place connectives $\lor$ and $\to$, i.e., $\Sigma^{(1)}=\{\neg,\Box\}$,
 $\Sigma^{(2)}=\{\lor,\to\}$ 
and $\Sigma^{(k)}=\emptyset$ if $k\neq 2$, and consider the $4$-valued $\Sigma$-Nmatrix $\KMt=\tuple{\{F,f,t,T\},\{t,T\},\cdot_{\KMt}}$ introduced by Kearns~\cite{kearns1981modal}:

 \begin{center} \begin{tabular}{c | c c c c}
${\lor}_\KMt$ & $F$ & $f$ & $t$ & $T$  \\ 
\hline
$F$&  $F$ & $f$&   $t$  & $T$ \\
$f$ &$f$ & $f$ &  $t,T$ & $T$ \\
$t$&  $t$ & $t,T$ & $t,T$ & $T$ \\
$T$ &$T$ & $T $&  $T$ &$T$ 
\end{tabular}
\quad
\begin{tabular}{c | c c c c}
${\to}_\KMt$ & $F$ & $f$ & $t$ & $T$  \\
\hline
$F$&  $T$ & $T$&   $T$  & $T$ \\
$f$ &$t$ & $t,T$ &  $t,T$ & $T$ \\
$t$&  $f$ & $f$ & $t,T$ & $T$ \\
$T$ &$F$ & $f $&  $t$ &$T$ 
\end{tabular}
\quad
\begin{tabular}{c | c | c}
 & ${\neg}_\KMt$ & ${\Box}_\KMt$ \\ 
\hline
$F$& $ T $ & $F,f$\\
$f$ &$ t $ & $F,f$\\
$t$& $ f $ & $F,f$\\
$T$ &$ F $ & $ t,T$
\end{tabular}
\end{center}
Note that $[p\to q]_\KMt=[\neg p \lor q]_\KMt$. %
First, we illustrate the fact that the multi-function induced by a formula cannot be reduced to the composition of the multi-functions of immediate subformulas. Let $A=p\lor \neg p$ and $B=\neg p\lor  p$.
The truth-table of $\lor_\KMt$ is symmetric and thus
 $[A]_\KMt=[B]_\KMt$.  However, $[\Box A\to \Box A]_\KMt\neq
 [\Box A\to \Box B]_\KMt$.
Looking at the diagonal of $\to_\KMt$ we conclude that $[\Box A\to \Box A]_\KMt(x)\subseteq\{t,T\}$ for every $x\in \{F,f,t,T\}$.
However, for instance
picking $v\in\Val(\KMt)$ extending
\begin{center}
\begin{tabular}{c|c|c|c|c|c|c|c|c} %
 & $p$ & $\neg p$  & $A$  & $B$  & $\square A$  & $\square B$  & $\square A\to \square A$  & $\square A\to \square B$ \\
\hline
$v$& $ t $ & $ f $ & $ T $ & $ t $ & $ T $ & $ F $ & $ T $& $ F $\\
 \end{tabular}
\end{center}
which is a prevaluation since
$f\in\neg_\KMt(t)$, 
$t,T\in(t \lor_\KMt f)=(f \lor_\KMt t)$, 
$T\in\neg_\KMt(T)$, 
$F\in\neg_\KMt(t)$,
$T \in (T \to_\KMt T)$, and
$T\in (T \to_\KMt F)$.
Hence, $F\in [\Box A\to \Box B]_\KMt(t)$,
and thus $[\Box A\to \Box A]_\KMt\neq [\Box A\to \Box B]_\KMt$.
\smallskip

We can also check that the logic defined by $\KMt$ fails to be finitely-determined. For $n\in\nats$, 
 let  
 $A_n=p_1\to (p_2 \to\dots (p_{n-1}\to p_n)\dots)$, and
 $\mathsf{Subst}_n$  
 collect every
  substitution $\tau:P_{n}\to P_{n+1}$.
Consider $\Gamma_n=\{A_n^\tau:\tau\in \mathsf{Subst}_n\}\setminus \{A_n\}$.
We have that
$\square\Gamma_n\not\sder_\KMt \square A_n$, as witnessed by the prevaluation defined by $v_n(B)=t$ for 
$B\in (\sub(\Gamma_n)\setminus \Gamma_n)\cup (\sub(A_n)\setminus \{A_n\})$,
$v_n(B)=v_n(\square B)=T$  for $B\in \Gamma_n$,
$v(A_n)=t$
and 
$v(\square A_n)=F$. 
 Clearly, $v_n(\square\Gamma_n)=\{T\}\subseteq\{t,T\}$ and $v_n(\square A_n)=F\notin\{t,T\}$.

On the other hand, for every
$\sigma:P_{n+1}\to P_n$ we have that
$\square\Gamma_n^{\sigma}\der_\KMt \square A_n^{\sigma}$, simply because
 $A_n^{\sigma}\in \Gamma_n^{\sigma}$. Indeed, given $\sigma:P_{n+1}\to P_n$, it suffices to show that 
        there is $\tau\in  \mathsf{Subst}_n$ such that 
        $\tau(p_i)\neq p_i$ for some $1\leq i\leq n$, and
        $(A_n)^{\sigma\circ \tau}=A_n^\sigma$,
 that is, that $\sigma\circ \tau=\sigma$.
By the pigeonhole principle we know that for any
 $\sigma:P_{n+1}\to P_n$ we must have that $\sigma(p_i)=\sigma(p_j)$ for some $1\leq i<j \leq n+1$.

Consider
$$\tau(p_\ell)=
\begin{cases}
 p_\ell &\mbox{ if }\ell\neq i\\
 p_j &\mbox{ if } \ell=i
\end{cases}$$
Note that $\tau(p_i)\neq p_i$.
Further,
$\tau(p_\ell)= p_\ell$ for $\ell\neq i$.
Thus, $\sigma(p_\ell)=\sigma(\tau(p_\ell))$ for $\ell\neq i$.
Finally, by assumption we have that $\tau(p_i)=p_j$, and since $\sigma(p_i)=\sigma(p_j)$ we conclude also that
$\sigma(p_i)=\sigma(\tau(p_i))$.
\hfill$\triangle$

\end{example}
}

\subsection{Computational problems}\label{subcompprob}

Given a class $\mathcal{M}$ of finite Nmatrices, we are interested in the following computational problems.

\begin{itemize}
\item ${\mathsf{Csq}}(\mathcal{M})$: given a finite signature $\Sigma$, a finite $\Sigma$-Nmatrix $\Mt\in\mathcal{M}$ and a finite set $\Gamma\cup\{A\}\subseteq L_\Sigma(P)$ determine whether $\Gamma\sder_\Mt A$; 
\item ${\mathsf{\exists Thm}}(\mathcal{M})$: given a finite signature $\Sigma$ and a finite $\Sigma$-Nmatrix $\Mt\in\mathcal{M}$ determine whether $\Thm(\Sigma,\sder_\Mt)\neq\emptyset$; 
\item ${\mathsf{Eqv}}(\mathcal{M})$: given a finite signature $\Sigma$ and finite $\Sigma$-Nmatrices $\Mt_1,\Mt_2\in\mathcal{M}$ determine whether $\sder_{\Mt_1}{=}\sder_{\Mt_2}$. 
\end{itemize}

We dub ${\mathsf{Csq}}(\mathcal{M})$ the \emph{consequence problem for class $\mathcal{M}$},  ${\mathsf{\exists Thm}}(\mathcal{M})$ the \emph{theorem existence problem for class $\mathcal{M}$}, and ${\mathsf{Eqv}}(\mathcal{M})$ the \emph{equivalence problem for class $\mathcal{M}$}. Despite the generality, we will typically consider the cases when ${\mathcal{M}}=\mathsf{NMatr}$ is the class of all finite Nmatrices, or when ${\mathcal{M}}=\mathsf{Matr}$ is the class of all finite matrices. As usual, the dual of a computational problem is the computational problem obtained by negating the envisaged condition. Given a computational problem $\mathsf{P}$ we will denote its \emph{dual} by $\overline{P}$. For instance, $\overline{{\mathsf{Eqv}}(\mathcal{M})}$ is the problem of determining whether $\sder_{\Mt_1}{\neq}\sder_{\Mt_2}$, given a finite signature $\Sigma$ and finite $\Sigma$-Nmatrices $\Mt_1,\Mt_2\in\mathcal{M}$.

A problem is \emph{decidable} if there is an algorithm that on each suitable input answers 
$\mathsf{yes/no}$ depending on whether the envisaged condition holds. 
Similarly, a problem is \emph{semidecidable} (also known as \emph{recursively enumerable}) if there is an algorithm that answers $\mathsf{yes}$ precisely on inputs  
for which the condition holds. 

Of course, if any of the above problems is (semi)decidable with respect to a class $\mathcal{M}$ then it is also (semi)decidable with respect to any class $\mathcal{M}'\subseteq\mathcal{M}$. Recall also from any basic textbook on computability theory (e.g,~\cite{sipser}) that  a problem $\mathsf{P}$ is decidable if and only if $\overline{\mathsf{P}}$ is decidable if and only if both  $\mathsf{P}$ and  $\overline{\mathsf{P}}$ are semidecidable, as well as the notion of \emph{computable reduction} (or many-one reduction) between problems, denoted by $\leq$, and the fact that if $\mathsf{P}\leq \mathsf{Q}$ and problem $\mathsf{Q}$ is (semi)decidable then also $\mathsf{P}$ is (semi)decidable.\smallskip

The problem  ${\mathsf{Csq}}(\mathsf{Matr})$ is known to be decidable (and in $\mathbf{coNP}$), namely via the method of truth-tables. A result of~\cite{Baaz2013} shows that the problem  ${\mathsf{Csq}}(\mathsf{NMatr})$ is also decidable (and still in $\mathbf{coNP}$). Taking these facts into account, as well as the compactness, finite-determinedness and local tabularity of the logic of any given finite-valued matrix, one gets close to understanding why the problem ${\mathsf{Eqv}}(\mathsf{Matr})$ is decidable. Consider the following very simple useful property of Tarskian logics.

\begin{proposition}\label{inclusion}
Let $\Sigma$ be a signature, $\tuple{\Sigma,\vdash_1}$ a compact logic, and $\tuple{\Sigma,\vdash_2}$ a $n$-determined logic for $n\in\nats$. Assuming 
$\Theta\subseteq L_\Sigma(P_n)$ is such that for each $A\in L_\Sigma(P_n)$ there exists $A^*\in\Theta$ with $A\dashv\sder_1 A^*$ and $A\dashv\sder_2 A^*$, then
 $$\vdash_1{\subseteq}\vdash_2$$
 $$\textrm{if and only if}$$ 
 $$\textrm{$\Gamma\vdash_1 A$ implies $\Gamma\vdash_2 A$ for all finite $\Gamma\cup\{A\}\subseteq \Theta$.}$$
\end{proposition}
\proof{If $ {\,\vdash_1} \subseteq {\vdash_2} $ then $\Gamma\vdash_1 A$ implies $\Gamma\vdash_2 A$ for all $\Gamma\cup\{A\}\subseteq L_\Sigma(P)$. We prove the converse implication, assuming that $\Gamma\vdash_1 A$ implies $\Gamma\vdash_2 A$ for all finite $\Gamma\cup\{A\}\subseteq \Theta$, and also that for some $\Delta\cup\{B\}\subseteq L_\Sigma$ we have $\Delta\vdash_1 B$.
From the compactness of $\tuple{\Sigma,\vdash_1}$, there exists finite $\Delta_0\subseteq \Delta$ such that $\Delta_0\vdash_1 B$. Using substitution-invariance, then, for every substitution $\sigma:P\to P_n$ it follows that $\Delta_0^\sigma\vdash_1 B^\sigma$. Hence, given that $\Delta_0^\sigma\cup\{B^\sigma\}\subseteq L_\Sigma(P_n)$ is finite, it follows that $(\Delta_0^\sigma)^*\cup\{(B^\sigma)^*\}\subseteq \Theta$ is also finite, and using the assumptions as well as monotonicity and transitivity, we have that $(\Delta_0^\sigma)^*\sder_1(B^\sigma)^*$, and so $(\Delta_0^\sigma)^*\sder_2(B^\sigma)^*$, and finally also $\Delta_0^\sigma\sder_2 B^\sigma$ for every $\sigma:P\to L_\Sigma(P_n)$. Therefore, since $\tuple{\Sigma,\vdash_2}$ is $n$-determined, it follows that $\Delta_0\sder_2 B$, and using monotonicity, it is also the case that $\Delta\sder_2 B$.\qed}\\

Take two $\Sigma$-matrices $\Mt_1$ and $\Mt_2$. Clearly, $\sder_{\Mt_1}{=}\sder_{\Mt_2}$ if and only if $\sder_{\Mt_1}{\subseteq}\sder_{\Mt_2}$ and $\sder_{\Mt_2}{\subseteq}\sder_{\Mt_1}$.
At the light of Propositions~\ref{finval} and \ref{inclusion}, one could verify the inclusion of each of the logics in the other by testing only formulas in the set $\Theta\subseteq L_\Sigma(P_n)$ where $n$ is an upper bound on the numbers of truth-values in each of the matrices. To make this procedure effective, we just need to show that there exists a finite set $\Theta$ with the desired property, which can be easily achieved using jointly the local tabularity of the logics of both Nmatrices.

\begin{proposition}\label{fintheta}
Let $\Sigma$ be a signature, $\Mt_1=\tuple{V_1,D_1,\cdot_{\Mt_1}}$ and $\Mt_2=\tuple{V_2,D_2,\cdot_{\Mt_2}}$ finite $\Sigma$-matrices. For every $n\in\nats$ with $|V_1|,|V_2|\leq n$ it is possible to compute a finite set $\Theta\subseteq L_\Sigma(P_n)$ such that for each $A\in L_\Sigma(P_n)$ there exists $A^*\in\Theta$ with $A\dashv\sder_1 A^*$ and $A\dashv\sder_2 A^*$.
\end{proposition}
\proof{According to Proposition~\ref{finval},  both $\tuple{\Sigma,\sder_{\Mt_1}}$ and $\tuple{\Sigma,\sder_{\Mt_2}}$ are $n$-determined, and thus there are finite sets $\Delta_{1,n},\Delta_{2,n}\subseteq L_\Sigma(P_n)$ such that for every $A\in L_\Sigma(P_n)$ there exist $A^{*_1}\in\Delta_{1,n}$ and $A^{*_2}\in\Delta_{2,n}$ with $A\dashv\sder_{\Mt_1}A^{*_1}$ and $A\dashv\sder_{\Mt_2}A^{*_2}$. However, it may well happen that $A^{*_1}\neq A^{*_2}$ and  $A\not{\!\!\!\dashv\sder_{\Mt_2}}A^{*_1}$, or $A\not{\!\!\!\dashv\sder_{\Mt_1}}A^{*_2}$, or both.

 To overcome this difficulty, we note that $\{([A]_{\Mt_1},[A]_{\Mt_2}):A\in L_\Sigma(P_n)\}$ is also a finite set, and that if $[A]_{\Mt_1}=[B]_{\Mt_1}$ and $[A]_{\Mt_2}=[B]_{\Mt_2}$ then it is immediate that $A\dashv\sder_{\Mt_1}B$ and $A\dashv\sder_{\Mt_2}B$. Thus, we need to compute $\Theta$ in such a way that it contains one representative formula for each of the finitely many pairs of $n$-place functions expressible in $\Mt_1$ and $\Mt_2$. For the purpose, let 
 $\Delta^{\mathcal{F}}=\{([A]_{\Mt_1},[A]_{\Mt_2}):A\in \Delta\}$ for any given set $\Delta\subseteq L_\Sigma(P_n)$, and define
 ${\Theta}_0=P_n$,  
${\Theta}_{i+1}={\Theta}_i\cup\{\conn(A_1,\dots,A_k): k\in\nats_0, \conn\in\Sigma^{(k)}, A_1,\dots,A_k\in\Theta_i,([\conn(A_1,\dots,A_k)]_{\Mt_1},[\conn(A_1,\dots,A_k)]_{\Mt_2})\notin\Theta_i^{\mathcal{F}}\}$. Just let $\Theta$ be the fixedpoint that the construction necessarily reaches in finite time.\qed}

\begin{corollary}
The problem ${\mathsf{Eqv}}(\mathsf{Matr})$ is decidable.
\end{corollary}

We shall not go into  the details here, but is clear that a similar (even simpler) line of thought allows us to conclude that ${\mathsf{\exists Thm}}(\mathsf{Matr})$ is also decidable. Due to substitution-invariance, it is clear that if a logic has a theorem then it has a theorem with (at most) a single variable. Thus, one can iteratively build a finite set of formulas representing each of the finitely many $1$-place functions expressible in the matrix, as above, while checking at each step whether any of the obtained formulas is a theorem.\smallskip

It is worth noting at this point that the ideas underlying these decidability results simply do not extend to proper Nmatrices, as noted also in~\cite{CZE}. Namely, as illustrated in Examples~\ref{unconstrained} and~\ref{something}, despite the fact that the logics of finite Nmatrices are still known to be compact (see~\cite{Avron:Lev:NDMVS}), in general they may fail to be finitely-determined and locally tabular, and in particular formulas expressing the same multi-function are typically not interderivable.

\section{Undecidability of $\mathsf{\exists Thm}(\mathsf{NMatr})$}\label{sec3}

We recall the undecidability of the theorem existence problem for Nmatrices. The proof consists of a computable reduction from a well known undecidable problem related to counter machines. We refer the reader to~\cite{FilipeNCL}, where the construction was first presented, for any additional details.\\

Recall from~\cite{minsky} that a \emph{counter-machine} is a tuple $\mathcal{C} = \tuple{n, Q, \qi, \delta}$, where $n \in \mathbb{N}$ is the number of \emph{counters},
    $Q$ is a finite set of \emph{states}, $\qi \in Q$ is the \emph{initial state}, and $\delta : Q \to (\{\incr{i}: 1 \leq i \leq n\}\times Q)
    \cup (\{\test{i}: 1 \leq i \leq n\}\times Q^2)$ is a partial \emph{transition} function. 
    The set  of \emph{halting states} of $\mathcal{C}$ is $\mathsf{H}_\mathcal{C} = \{q \in Q : \delta(q) \text{ undefined}\}$.
The set of \emph{configurations} of $\mathcal{C}$ is defined as $\mathsf{Config}_{\mathcal{C}} = Q \times \nats_0^n$, whereas configurations in $\{\qi\}\times \nats_0^n$ are dubbed \emph{initial}, and configurations in $\mathsf{HConfig}_{\mathcal{C}}=\mathsf{H}_\mathcal{C}\times \nats_0^n$ are dubbed \emph{halting}. 

Counter-machines compute deterministically, from a given initial configuration until, possibly, a halting configuration is reached, according to the progress function $\nxt_\mC : (\mathsf{Config}_{\mC}\setminus\mathsf{HConfig}_{\mathcal{C}}) \to \mathsf{Config}_{\mC}$ defined below, where $\vec{e_i}=(e_{i,1},\dots,e_{i,n}) \in \mathbb{N}_0^n$ is such that $e_{i,i} = 1$, and $e_{i,j} = 0$ for  $j \neq i$, $q,q',q''\in Q$ and $\vec{a}=(a_1,\dots,a_n)\in\nats_0^n$.
\begin{equation*}
    \nxt{}_\mC(q,\vec{a}) =
    \begin{cases}
        \tuple{q',\vec{a} + \vec{e_i}} &\text{if } \delta(q) = (\incr{i}, q'), \\
        \tuple{q',\vec{a} - \vec{e_i}} &\text{if } \delta(q) = (\test{i}, q',q'') \text{ and } a_i \neq 0, \\
        \tuple{q'',\vec{a}} &\text{if } \delta(q) = (\test{i}, q',q'') \text{ and } a_i = 0, \\
    \end{cases}
\end{equation*}

Computations of $\mathcal{C}$ are thus non-empty finite, or infinite, sequences of configurations $\tuple{C_0,C_1,\dots, C_k,\dots}$ such that $C_0$ is initial, and for each $k\in\nats_0$ either $C_k$ is halting and is precisely the last configuration of the (finite) sequence, or else
 $C_{k+1} = \nxt_\mC(C_{k})$ is the next configuration of the sequence. Henceforth, if $\vec{a}\in\nats_0^n$ determines the initial values of the counters, we denote by $\mathsf{comp}(\mC,\vec{a})$ the (finite or infinite) computation of $\mC$ starting at  $C_0=\tuple{\qi,\vec{a}}$. It is well known from~\cite{minsky} that assuming all counters are initially set to zero does not hinder the Turing-completeness of the model, and so that the following problem is undecidable. 
\begin{itemize}
\item ${\mathsf{HALT}}$: given a counter machine $\mC$, determine whether $\mC$ halts when all counters are initially set to zero, i.e., whether $\mathsf{comp}(\mC,\vec{0})$ is finite.
\end{itemize}

The following construction, and results, were originally presented in~\cite{FilipeNCL}, where the reader may find further details.


\begin{definition}\label{def-mach2matr}
    Let $\mC = \tuple{n,Q,\qi,\delta}$ be a counter-machine.
    
    The \emph{signature induced by $\mC$} is $\Sigma_{\mC}$, defined by 
    $\Sigma_{\mC}^{(0)} = \{\zero,\epsilon\}$, $\Sigma_{\mC}^{(1)} = \{\suc\}$, $\Sigma_{\mC}^{(n+1)} = \{\step^q : q \in Q\}$, and $\Sigma^{(j)}_{\mC} = \emptyset$ for $j \notin \{0,1,n+1\}$.\smallskip

   The \emph{$\Sigma_{\mC}$-Nmatrix induced by $\mC$} is $\Ct = \tuple{V_\mC,D_\mC,\cdot_{\Ct}}$ defined by
\begin{itemize}
\item $V_\mC=\Rm \cup \Conf \cup  \{\mathrm{init},\mathrm{error}\}$,
\item $D_\mC=\HConf$, and
\item for all $q,q',s \in Q$, $x\in V_\mC$ and $\vec{y},\vec{z}\in V_\mC^n$,\\[-11mm]

\begin{minipage}{4cm}
\mbox{}\\
\noindent\begin{equation*}
        \zero_{\Ct}  = \{\n_{= 0},\n_{\geq 0}\} 
    \end{equation*}
    \vspace*{1mm}
    \begin{equation*}
        \epsilon_{\Ct} = \{\mathrm{init}\}
    \end{equation*}
\end{minipage}
\begin{minipage}{5cm}
\mbox{}\\[4mm]
    \begin{equation*}
        \suc_{\Ct}(x) =
        \begin{cases}
            \{\n_{\geq 1}\} &\quad\text{if } x = \n_{= 0}, \\
            \{\n_{\geq 0},\n_{\geq 1}\} &\quad\text{if } x = \n_{\geq 0}, \\
            \{\n_{\geq 2}\} &\quad\text{if } x \in\{ \n_{\geq 1},\n_{\geq 2}  \},\\
            \{\mathrm{error}\} &\quad\text{otherwise,}
        \end{cases}
    \end{equation*}
\end{minipage}

    \begin{equation*}
        \step^q_{\Ct}(x,\vec{z}) =
        \begin{cases}
            \{\mathrm{conf}_{q,\vec{z}}\}
            &\text{if } x=\mathrm{init},q=q_{\mathsf{init}} \\
            &\hspace{40pt} \text{ and } \vec{z}\in \{\n_{= 0}\}^n\cup\{\n_{\geq 0}\}^n, \text{ or} \\
            &\text{if } x=\mathrm{conf}_{q',\vec{y}},\,  \vec{z} \in \Rm^n,  \text{ and }\\
            &\qquad\, \delta(q')=(\inc(i),q),z_i\in \suc_{\Ct}(y_i) \\
            &\hspace{40pt}\text{ and } z_\ell=y_\ell  \text{ for }\ell\neq i,\text{ or} \\
            &\qquad\, \delta(q')=(\test{i},q,s),y_i\in \suc_{\Ct}(z_i) \\
            &\hspace{40pt}\text{ and } z_\ell=y_\ell  \text{ for }\ell\neq i,\text{ or} \\
            &\qquad\, \delta(q')=(\test{i},s,q),y_i\in\zero_{\Ct} \\
            &\hspace{40pt}\text{ and }\vec{z}=\vec{y},\\
            \{\mathrm{error}\} &\text{otherwise,}
            \end{cases}
    \end{equation*}

\end{itemize}
where $\Rm=\{\n_{= 0},\n_{\geq 0},\n_{\geq 1},\n_{\geq 2}\}$, $\Conf=\{\mathrm{conf}_{q,\vec{\n}} :q \in Q,\vec{\n}\in \Rm^n\}$ and $\HConf=\{\mathrm{conf}_{q,\vec{\n}}\in\Conf:q\in\mathsf{H}_\mathcal{C}\}$.\hfill$\triangle$
\end{definition}

Note that $\Sigma_{\mathcal{C}}$ is a finite signature and $\Ct$ is finite-valued, given any counter machine $\mathcal{C}$.
$\Sigma_\mC$ is conceived for encoding finite sequences of configurations of $\mC$ as closed formulas. 
Natural numbers are encoded via $\enc : \mathbb{N}_0 \to L_{\Sigma_{\mC}}(\emptyset)$ defined by
    $\enc(0) = \zero$, and $\enc(a) = \suc(\enc(a-1))$ for $a > 0$.  
    Given $\vec{a}=(a_1,\dots,a_n)\in\nats_0^n$, we will write $\enc(\vec{a})$ to denote $(\enc(a_1),\dots,\enc(a_n))$.    
    Finite sequences of configurations are encoded via $\seq : \mathsf{Config_{\mC}^*} \to L_{\Sigma_{\mC}}(\emptyset)$ defined
     by $\seq(\tuple{}) = \epsilon$, and $\seq(\tuple{C_0, \dots, C_k, \tuple{q,\vec{a}}}) = \step^q(\seq(\tuple{C_0, \dots, C_k}),\enc(\vec{a}))$.
     
$\Ct$ is conceived for closely mimicking the computation of $\mC$ when all counters are initially set to zero, despite the `abstract' interpretation given to natural numbers. Technically, the value $\mathrm{error}\notin D_\mC$ is absorbing, that is, if $A\in L_{\Sigma_\mC}(P)$, $B\in\Sub(A)$ and $v\in\Val(\Ct)$ is such that $v(B)=\mathrm{error}$ then also $v(A)=\mathrm{error}$. This fact immediately implies that any theorem of $\tuple{\Sigma_\mC,\vdash_\Ct}$ must necessarily be a closed formula. The definition of $\Ct$ further guarantees that if $\Var(A)=\emptyset$, if $v(A)\in \Conf$ then it must be the case that 
$A=\seq(\vec{C})$ for some finite non-empty sequence $\vec{C}=\tuple{C_0, \dots, C_k}$ of configurations. Easily, also, if $v(A)\in \HConf$ then it must be the case that $C_k$ is a halting configuration.

Given that the non-determinism in $\Ct$ is confined to $\zero_{\Ct}$ and $\suc_{\Ct}$, $v(A)$ is completely determined by $v(\enc(\nats_0))$ for every $v\in\Val(\Ct)$ and every closed $A\in L_{\Sigma_\mC}(\emptyset)$.
The following lemma allows us to navigate the different possibilities, and their consequences.

\begin{lemma}\label{lemma-mach2matr}
Let $\mC$ be a counter machine and $v\in\Val(\Ct)$. The restriction of $v$ to $\enc(\nats_0)$ is such that $v_{|\enc(\nats_0)}\in\{v_=,v_\omega\}\cup\{v_k:k\in\nats_0\}$ where
\begin{itemize}
\item $v_=(\enc(0))=\n_{= 0}$, $v_=(\enc(1))=\n_{\geq 1}$, and $v_=(\enc(a))=\n_{\geq 2}$ for all $a>1$,
\item $v_\omega(\enc(a))=\n_{\geq 0}$ for all $a\in\nats_0$, and
\item $v_k(\enc(a))=\n_{\geq 0}$ for all $a\leq k$, $v_k(\enc(k+1))=\n_{\geq 1}$, and $v_k(\enc(a))=\n_{\geq 2}$ for all $a>k+1$.
\end{itemize}
Further, for $a,b\in\nats_0$ the following properties hold.
\begin{enumerate}[(i)]
\item If $a\neq b+1$ then there is $v\in\Val(\Ct)$ such that $v(\enc(a))\notin\suc_\Ct(v(\enc(b))).$
\item If $a\neq b$ then there is $v\in\Val(\Ct)$ such that $v(\enc(a))\neq v(\enc(b))$.
\end{enumerate}
\end{lemma}
\proof{If $v\in\Val(\Ct)$ is such that $v(\enc(0))=v(\zero)=\n_{=0}$ then $v(\enc(1))=\n_{\geq 1}$, and $v(\enc(a))=\n_{\geq 2}$ for all $a>1$, which makes $v_{|\enc(\nats_0)}=v_=$. Otherwise, 
$v(\enc(0))=\n_{\geq 0}$ and $\suc_\Ct(\n_{\geq 0})=\{\n_{\geq 0},\n_{\geq 1}\}$. Thus, either $v(\enc(a))=\n_{\geq 0}$ for all $a\in\nats_0$, in which case $v_{|\enc(\nats_0)}=v_\omega$, or else 
for some $k\in\nats_0$ we have $v(\enc(a))=\n_{\geq 0}$ for all $n\leq k$, and then $v(\enc(k+1))=\n_{\geq 1}$ and $v(\enc(a))=\n_{\geq 2}$ for all $a>k+1$, in which case $v_{|\enc(\nats_0)}=v_k$.

\begin{enumerate}[{\it (i)}]
\item If $a<b$ take $v_a$, if $a=b=0$ take $v_=$, if $a=b\neq 0$ take $v_{a-1}$, and if $a>b+1$ take $v_b$. In any case, extend the obtained prevaluation to $v\in\Val(\Ct)$.
\item Assuming, without loss of generality, that $a<b=(b-1)+1$, we have from case {\it (i)} that there exists $v$ such that $v(\enc(a))\notin\suc_\Ct(v(\enc(b-1)))$. Thus, $v(\enc(a))\neq v(\enc(b))\in\suc_\Ct(v(\enc(b-1)))$.
\qed
\end{enumerate}
}

We now show that the logic of $\Ct$ has at most one theorem, which is precisely the encoding of the computation of $\mC$ when all counters are initially set to zero, in case it is halting. Said another way, $\tuple{\Sigma_\mC,\Ct}$ has no theorems if $\mathsf{comp}(\mC,\vec{0})$ is infinite, and otherwise has exactly one theorem which is precisely $\seq(\mathsf{comp}(\mC,\vec{0}))$.

\begin{proposition}\label{mach2nmatr}
Let $\mathcal{C}$ be a counter machine. For every $A\in L_{\Sigma_{\mathcal{C}}}(P)$, we have 
$$\textrm{$\mathsf{comp}(\mC,\vec{0})$ is finite and $A=\seq(\mathsf{comp}(\mC,\vec{0}))$ \underline{if and only if} $A\in\Thm(\Sigma_{\mathcal{C}},\sder_\Ct)$. }$$
\end{proposition}
\proof{Easily, by induction, one shows that for every finite non-empty prefix $\vec{C}=\tuple{C_0,\dots,C_k}$ of $\mathsf{comp}(\mC,\vec{0})$, if $C_k=\tuple{q,\vec{x}}$ then we have $v(\seq(\vec{C}))=\mathrm{conf}_{q,v(\enc(\vec{x}))}$ for every $v\in\Val(\Ct)$.\smallskip

Assume that $\mathsf{comp}(\mC,\vec{0})$ is finite and let $v\in\Val(\Ct)$. Since the last configuration of $\mathsf{comp}(\mC,\vec{0})$ is halting it follows that $v(\seq(\mathsf{comp}(\mC,\vec{0})))\in\HConf=D_\mC$.\smallskip

Reciprocally, if $A\in\Thm(\Sigma_{\mathcal{C}},\sder_\Ct)$, we know that $A=\seq(\vec{C})$ for some finite non-empty sequence $\vec{C}=\tuple{C_0,\dots,C_k}$ of configurations, with $C_k$ halting. Thus, in order to 
show that $\vec{C}=\mathsf{comp}(\mC,\vec{0})$ it suffices to prove that $\vec{C}$ is a prefix of $\mathsf{comp}(\mC,\vec{0})$. We show, by induction on $j\leq k$ , that $\tuple{C_0,\dots,C_j}$ is a prefix of $\mathsf{comp}(\mC,\vec{0})$.

\begin{itemize}
\item If $C_0=\tuple{q,\vec{a}}$ and for $v\in\Val(\Ct)$ we have $v(\seq(\tuple{C_0}))\neq\mathrm{error}$ then, necessarily, $v(\seq(\tuple{C_0}))=\step^q_\mC(\mathrm{init},v(\enc(\vec{a})))=\mathrm{conf}_{q,v(\enc(\vec{a}))}$, and thus $q=\qi$ and $v(\enc(\vec{a}))\in\{\n_{= 0}\}^n\cup\{\n_{\geq 0}\}^n$; when $v$ extends $v_=$, in particular, it follows that $v(\enc(\vec{a}))\in\{\n_{= 0}\}^n$ and therefore $\vec{a}=\vec{0}$; we conclude that $C_0=\tuple{\qi,\vec{0}}$.
\item If $j>0$,  $C_{j-1}=\tuple{q,\vec{a}}$ is halting and $C_j=\tuple{q',\vec{b}}$ then, for $v\in\Val(\Ct)$, given that $v(\seq(\tuple{C_0,\dots,C_j}))\neq\mathrm{error}$ then, necessarily, 
$v(\seq(\tuple{C_0,\dots,C_j}))=\seq^{q'}_\Ct(\mathrm{conf}_{q,v(\enc(\vec{a}))},v(\enc(\vec{b})))=\mathrm{conf}_{q',v(\enc(\vec{b}))}$; this implies that $\delta(q)$ is defined and $q\not\in\mathsf{H}_\mathcal{C}$, and we are left with showing that $C_j=\nxt_\mC(C_{j-1})$; we need to consider three cases.

\begin{itemize}
\item If $\delta(q)=(\incr{i},q'')$, it follows that $q'=q''$, and for all $v\in\Val(\Ct)$ we have $v(b_i)\in\suc_\mC(v(a_i))$ and $v(b_\ell)=v(a_\ell)$ for $\ell\neq i$, and using Lemma~\ref{lemma-mach2matr}{\it (i-ii)} we conclude that $\vec{b}=\vec{a}+\vec{e_i}$.
\item If $\delta(q)=(\test{i},q'',q''')$ and $a_i\neq 0$, one has $v_=(\enc(a_i))\notin\zero_\Ct$ and thus $q'=q''$; even if $q''=q'''$, if for $v$ extending $v_\omega$, or extending some $v_c$ for $c\in\nats_0$, one has $v(\enc(a_i))\in\zero_\Ct$ and $v(\enc(\vec{a}))=v(\enc(\vec{b}))$, then $v(\enc(a_i))=v(\enc(b_i))=\n_{\geq 0}$, and since $\n_{\geq 0}\in\suc_\Ct(\n_{\geq 0})$ it is also the case that $v(a_i)\in\suc_\mC(v(b_i))$; it follows that for all $v\in\Val(\Ct)$ we have $v(a_i)\in\suc_\mC(v(b_i))$, and $v(b_\ell)=v(a_\ell)$ for $\ell\neq i$, and using Lemma~\ref{lemma-mach2matr}{\it (i-ii)} we conclude that $\vec{b}=\vec{a}-\vec{e_i}$.
\item If $\delta(q)=(\test{i},q'',q''')$ and $a_i= 0$, one has $v_=(\enc(a_i))=\n_{=0}\notin\suc_\Ct(v_=(\enc(b_i)))$ and thus $q'=q'''$; even if $q''=q'''$, if for $v$ extending $v_\omega$, or extending some $v_c$ for $c\in\nats_0$, one has $v(\enc(a_i))=\n_{\geq 0}\in\suc_\Ct(v(\enc(b_i)))$, then $v(\enc(a_i))=v(\enc(b_i))=\n_{\geq 0}\in\zero_\mC$; it follows that for all $v\in\Val(\Ct)$ we have $v(a_i)\in\zero_\mC$ and $v(\vec{b})=v(\vec{a})$, and using Lemma~\ref{lemma-mach2matr}{\it (ii)} we conclude that $\vec{b}=\vec{a}$.
\end{itemize}
In all three cases we conclude that $C_j=\nxt_\mC(C_{j-1})$, and $\vec{C}$ is a prefix of $\mathsf{comp}(\mC,\vec{0})$. Since $C_k$ is halting it follows that $\vec{C}=\mathsf{comp}(\mC,\vec{0})$.\qed
\end{itemize}
}

The following is a straightforward corollary of Proposition~\ref{mach2nmatr}.

\begin{corollary}\label{pre-halt2thm}
A counter machine $\mathcal{C}$ halts {when all counters are initially set to zero}  if and only if $\Thm(\Sigma_{\mathcal{C}},\sder_\Ct)\neq\emptyset$.
\end{corollary}

As it is clear that one can compute $\Sigma_{\mathcal{C}}$ and $\Ct$ directly from $\mathcal{C}$, the construction provides the following computable reduction.

\begin{corollary}\label{halt2thm}
$\mathsf{HALT}\leq{\mathsf{\exists Thm}(\textsf{NMatr})}$.
\end{corollary}

Our main result now follows, as a direct consequence of the {undecidability of $\mathsf{HALT}$~\cite{minsky}} and Corollary~\ref{halt2thm}.

\begin{theorem}\label{exthm}
$\mathsf{\exists Thm}(\textsf{NMatr})$ is undecidable.
\end{theorem}

\section{Undecidability of $\mathsf{Eqv}(\mathsf{NMatr})$}\label{sec4}

In this section we present our main result, the undecidability of the equivalence problem for Nmatrices. The proof shall consist of a computable reduction from the theorem existence problem. The reduction is based on a surprisingly simple construction on Nmatrices, which bears similarities to results on \emph{infected} Nmatrices such as those presented in~\cite{ISMVL}. The construction essentially consists in adding to a given Nmatrix a designated copy of each of its undesignated truth-values.\smallskip

\begin{definition}\label{def-infect}
Let $\Sigma$ be a signature and $\Mt = \tuple{V,D,\cdot_\Mt}$ a $\Sigma$-Nmatrix. The \emph{tilded} Nmatrix $\widetilde{\Mt} = \tuple{\widetilde{V},,\widetilde{D},\cdot_{\widetilde{\Mt}}}$ is defined by 
\begin{itemize}
\item $\widetilde{V}=V\cup\{\widetilde{x}:x\in V\setminus D\}$,
\item $\widetilde{D}=D\cup\{\widetilde{x}:x\in V\setminus D\}$, and
\item for all $k\in\nats_0$, $\conn\in\Sigma^{(k)}$ and $x_1,\dots,x_k\in\widetilde{V}$, $$\conn_{\widetilde{\Mt}}(x_1,\dots,x_k)=u^{-1}(\conn_\Mt(u(x_1),\dots,u(x_k))),$$ 
\end{itemize}
where $u:\widetilde{V}\to V$ is the \emph{tilde-forgetful} function with $u(x)=u(\widetilde{x})=x$ for $x\in V$.\hfill$\triangle$%
\end{definition}

Note that $\widetilde{\Mt}$ is finite-valued precisely when $\Mt$ is finite-valued. Despite its simplicity, this innocent looking \emph{tildeing}  construction has an interesting meaning, as illustrated by the following examples.

\begin{example}{(Tilded Nmatrices)}\label{tilded}\smallskip
\begin{itemize}
\item 
Let $\Sigma$ be the signature with a single 1-place connective $\flat$ of Example~\ref{unarys}. Recall also the $\Sigma$-matrix $\Mt_7=(\{0,1\},\{1\},\cdot_{\Mt_7})$, as well as its  tilded version, the three-valued $\Sigma$-Nmatrix 
$\widetilde{\Mt}_7=(\{0,\widetilde{0},1\},\{\widetilde{0},1\},\cdot_{\widetilde{\Mt}_7})$, as defined below.\smallskip

\hspace*{3cm}
\begin{minipage}{3cm}
\begin{center}
\begin{tabular}{C{5pt}|C{5mm}}
 & ${\flat_{\Mt_7}}$\\[1mm]
\hline
$0$& $ 1 $\\
$1$ &$ 0 $
\end{tabular}
\end{center}
\end{minipage}
\begin{minipage}{3cm}
\begin{center}
\begin{tabular}{C{5pt}|C{5mm}}
 & ${\flat_{\widetilde{\Mt}_7}}$\\[1mm]
\hline
$0$& $ 1 $\\
$\widetilde{0}$ & $ 1 $\\
$1$ &$ 0,\widetilde{0} $
\end{tabular}
\end{center}
\end{minipage}\mbox{}\\

It is clear that $\tuple{\Sigma,\sder_{\Mt_7}}$ is the negation-only fragment of classical logic, and can be axiomatized by the following schematic rules (see~\cite{rautenberg}).

$$\frac{p}{\;\flat\flat p\;}\qquad\frac{\;\flat\flat p\;}{p}\qquad\frac{\;p,\flat p\;}{q}$$

In particular $\Thm(\Sigma,\sder_{\Mt_7})=\emptyset$, and thus, as we will show below in Proposition~\ref{infect}, $\widetilde{\Mt}_7$ is equivalent to $\Ut_\Sigma$. For instance, $\{A\}\not\vdash_{\widetilde{\Mt}_7}\flat\flat A$, as can be witnessed by any valuation $\widetilde{v}\in\Val(\widetilde{\Mt}_7)$ such that $\widetilde{v}(A)=\widetilde{0}$, $\widetilde{v}(\flat A)=1$, and $\widetilde{v}(\flat\flat A)=0$.\smallskip

\item 
Now, let $\Sigma_\limp$ be the signature with a single 2-place connective $\limp$, and consider the $\Sigma_\limp$-matrix $\It=(\{0,1\},\{1\},\cdot_{\It})$, as well as its  tilded version, the three-valued $\Sigma_\limp$-Nmatrix 
$\widetilde{\It}=(\{0,\widetilde{0},1\},\{\widetilde{0},1\},\cdot_{\widetilde{\Mt}_7})$, as defined below.\smallskip

\hspace*{3cm}
\begin{minipage}{3cm}
\begin{center}
\begin{tabular}{C{11pt}|C{4.5mm}|C{4.5mm}}
${\limp_\It}$ & $0$ & $1$\\[1mm]
\hline
$0$& $ 1 $ & $ 1 $\\
$1$ &$ 0 $ & $ 1 $
\end{tabular}
\end{center}
\end{minipage}\qquad
\begin{minipage}{3cm}
\begin{center}
\begin{tabular}{C{11pt}|C{4.5mm}|C{4.5mm}|C{4.5mm}}
${\limp_{\,\widetilde{\It}}}$ & $0$ & $\widetilde{0}$ & $1$\\[1mm]
\hline
$0$& $ 1 $ & $ 1 $ & $ 1 $\\
$\widetilde{0}$& $ 1 $ & $ 1 $& $ 1 $\\
$1$ &$ 0,\widetilde{0} $ & $ 0,\widetilde{0} $& $ 1 $
\end{tabular}
\end{center}
\end{minipage}
\mbox{}\\

It is clear that $\tuple{\Sigma_\limp,\sder_{\It}}$ is the implication-only fragment of classical logic, and can be axiomatized by the following schematic rules, or actually three familiar axioms and one proper rule, {modus ponens} (see~\cite{rautenberg}).

$$\frac{}{\;A\limp(B\limp A)\;}\qquad\frac{}{\;(A\limp(B\limp C))\limp((A\limp B)\limp(A\limp C))\;}$$\mbox{}\smallskip
$$\frac{}{\;((A\limp B)\limp A)\limp A\;}\qquad\qquad\frac{\;A,A\limp B\;}{B}$$\mbox{}\smallskip

$\Thm(\Sigma_\limp,\sder_{\It})\neq\emptyset$, and thus, as we will show below in Proposition~\ref{infect}, the tildeing construction retains all the theorems (not just the axioms). For instance, we have $(A\limp A)\in \Thm(\Sigma_\limp,\sder_{\,\widetilde{\It}})$, as for every $\widetilde{v}\in\Val(\widetilde{\It})$ we have that $(v(A)\limp_{\,\widetilde{\It}}v(A))={1}$. Still, other consequences are lost. Namely, modus ponens no longer holds, i.e., $\{A,A\limp B\}\not\vdash_{\,\widetilde{I}} B$, as can be witnessed by any valuation $\widetilde{v}\in\Val(\widetilde{\It})$ such that 
$\widetilde{v}(A)=1$, $\widetilde{v}(B)=0$, and $\widetilde{v}(A\limp B)=\widetilde{0}$.\hfill$\triangle$\smallskip
\end{itemize}
\end{example}

We now show that the tildeing construction indeed has the properties suggested by the previous examples. The construction always yields a Nmatrix whose logic is closely related with the logics of unconstrained Nmatrices, as defined in Example~\ref{unconstrained}. Namely, the logic of $\widetilde{\Mt}$ is \emph{quasi-unconstrained}, the exception being the theorems (if any) of the logic of the original Nmatrix $\Mt$. We start by proving a useful lemma.

\begin{lemma}\label{lemma-infect}
Let $\Sigma$ be a signature and $\Mt = \tuple{V,D,\cdot_\Mt}$ a $\Sigma$-Nmatrix. For every function $\widetilde{v}:L_\Sigma(P)\to\widetilde{V}$, $\widetilde{v}\in\Val(\widetilde{\Mt})$ if and only if $u\circ \widetilde{v}\in\Val(\Mt)$, where $u:\widetilde{V}\to V$ is the tilde-forgetful function of Definition~\ref{def-infect}.
\end{lemma}
\proof{Note that

$\widetilde{v}\in\Val(\widetilde{\Mt})$ iff

\quad${}$for all $k\in\nats_0,\conn\in\Sigma^{(k)}$ and $A_1,\dots,A_k\in L_\Sigma(P)$
 
\qquad$\widetilde{v}(\conn(A_1,\dots,A_k))\in\conn_{\widetilde{\Mt}}(\widetilde{v}(A_1),\dots,\widetilde{v}(A_k))$ iff
 
\qquad$\widetilde{v}(\conn(A_1,\dots,A_k))\in u^{-1}(\conn_{{\Mt}}(u(\widetilde{v}(A_1)),\dots,u(\widetilde{v}(A_k))))$ iff

\qquad$u(\widetilde{v}(\conn(A_1,\dots,A_k)))\in \conn_{{\Mt}}(u(\widetilde{v}(A_1)),\dots,u(\widetilde{v}(A_k)))$ iff
 
\qquad$(u\circ \widetilde{v})(\conn(A_1,\dots,A_k))\in \conn_{{\Mt}}((u\circ \widetilde{v})(A_1),\dots,(u\circ \widetilde{v})(A_k))$ iff

 $u\circ \widetilde{v}\in\Val({\Mt})$.
%
%
%
%
%
%
 \qed}\\

\begin{proposition}\label{infect}
Let $\Sigma$ be a signature and $\Mt$ a $\Sigma$-Nmatrix. 
For every $\Gamma\cup\{A\}\subseteq L_\Sigma(P)$, we have
$$\Gamma\sder_{\widetilde{\Mt}} A\textrm{ \underline{if and only if} } A\in\Gamma\textrm{ or }A\in\Thm(\Sigma,\sder_\Mt).$$
\end{proposition}
\proof{ Suppose that $\Gamma\sder_{\widetilde{\Mt}} A$, with $A\notin \Gamma$, and let $v\in\Val(\Mt)$. Consider the function $\widetilde{v}:L_\Sigma(P)\to\widetilde{V}$ defined as follows for each formula $B\in L_\Sigma(P)$.
 $$\widetilde{v}(B)=\left\{\begin{array}{cl}
                       v(B) & \textrm{if $B=A$}\\
                       x & \textrm{if $B\neq A$ and $v(B)=x\in D$}\\
                       \widetilde{x} & \textrm{if $B\neq A$ and $v(B)=x\notin D$}
                       \end{array}\right.$$         
 It is immediate from this definition that $u\circ \widetilde{v}=v$, and so $\widetilde{v}\in\Val(\widetilde{\Mt})$ according to Lemma~\ref{lemma-infect}. It is also straightforward to check that $\widetilde{v}(B)\in\widetilde{D}$ for every formula $B\in L_\Sigma\setminus\{A\}$, and thus $\widetilde{v}(\Gamma)\subseteq\widetilde{D}$. Given that $\Gamma\sder_{\widetilde{\Mt}} A$ it must also be the case that $\widetilde{v}(A)=v(A)\in \widetilde{D}$. However, $v(A)\in V$ so 
 $v(A)\in\widetilde{D}\cap V=D$, and we can conclude that $A\in\Thm(\Sigma,\sder_\Mt)$.\smallskip
 
Reciprocally, as $\sder_{\widetilde{\Mt}}$ always satisfies reflexivity, we need only consider the case when $A\in\Thm(\Sigma,\sder_\Mt)$. If $v\in\Val(\widetilde{\Mt})$ then Lemma~\ref{lemma-infect} guarantees $u\circ v\in\Val(\Mt)$, and thus it must be the case that $u(v(A))\in D$.
Therefore, we have that $v(A)\in u^{-1}(D)=D\subseteq \widetilde{D}$, and we can conclude that $\emptyset\;\sder_{\widetilde{\Mt}}A$, and by monotonicity $\Gamma\sder_{\widetilde{\Mt}}A$ for every $\Gamma\subseteq L_\Sigma(P)$.\qed}\\

The following is a straightforward corollary of Proposition~\ref{infect} and the characterization of the logic of unconstrained Nmatrices, as introduced in Example~\ref{unconstrained}.

\begin{corollary}\label{pre-thm2eqv}
Let $\Sigma$ be a signature and $\Mt$ a $\Sigma$-Nmatrix. We have
$\Thm(\Sigma,\sder_\Mt)\neq\emptyset$ if and only if $\widetilde{\Mt}$ and $\Ut_\Sigma$ are \underline{not} equivalent.
\end{corollary}

When we are working with a finite signature $\Sigma$ and a finite-valued $\Sigma$-Nmatrix $\Mt$, it is clear that one can compute $\Ut_\Sigma$ from $\Sigma$, and $\widetilde{M}$ from $\Mt$, which therefore provide the following computable reduction.

\begin{corollary}\label{thm2eqv}
$\mathsf{\exists Thm}(\textsf{NMatr})\leq\overline{\mathsf{Eqv}(\textsf{NMatr})}$.
\end{corollary}

Our main result now follows, as a direct consequence of Theorem~\ref{exthm}, Corollary~\ref{thm2eqv}, and the fact that a problem is as decidable as its dual.

\begin{theorem}\label{main}
$\mathsf{Eqv}(\textsf{NMatr})$ is undecidable.
\end{theorem}

\section{Nmatrix equivalence in practice}\label{sec5}

We have established that $\mathsf{Eqv}(\textsf{NMatr})$ is undecidable. Nevertheless, it is clear that $\overline{\mathsf{Eqv}(\textsf{NMatr})}$ is semidecidable. Indeed, if two given finite $\Sigma$-Nmatrices are not equivalent, we can always use brute-force to search for a finite set $\Gamma \cup \{A\} \subseteq L_\Sigma(P)$ such that either $\Gamma \vdash_{\Mt_1} A$ and $\Gamma \not\vdash_{\Mt_2} A$, or $\Gamma \vdash_{\Mt_2} A$ and $\Gamma \not\vdash_{\Mt_1} A$. This means that, in practice, showing that two Nmatrices define distinct logics can always be attained by computing/exhibiting an appropriate counterexample, as we did already in Examples~\ref{unarys} and~\ref{tilded}.\smallskip

 In the following subsections, we illustrate various techniques that can be used in order to prove that two finite Nmatrices define the same logic, despite the general undecidability of the problem.

\subsection{Axiomatizable logics}

The problem of determining whether two given Nmatrices are equivalent is considerably simplified when we have access to finite axiomatizations for the logics they define. 
As usual, a set of \emph{rules} $R=\{\frac{\,\Gamma_i\,}{A_i}:i\in I\}$ where each $\Gamma_i\cup\{A_i\}\subseteq L_\Sigma(P)$ \emph{axiomatizes} the logic $\tuple{\Sigma,\sder}$ where $\sder$ is the least consequence relation such that $\Gamma_i\sder A_i$ for each $i\in I$. Further, $R$ is said to be a \emph{finite axiomatization} when $I$ is a finite set, as well as every $\Gamma_i$.

\begin{proposition}\label{axprop}
Let $\Sigma$ be a signature, and $\tuple{\Sigma,\sder_1}$,   $\tuple{\Sigma,\sder_2}$ be logics. 
If $\tuple{\Sigma,\sder_1}$ is axiomatized by $R=\{\frac{\,\Gamma_i\,}{A_i}:i\in I\}$
then 
$$\text{$\sder_1{\subseteq} \sder_2$ if and only if  $\Gamma_i\sder_2 A_i$ for every $i\in I$.}$$
\end{proposition}
\begin{proof}
Just note that $\Gamma_i\sder_2 A_i$ for every $i\in I$ is precisely equivalent to $\sder_1{\subseteq} \sder_2$ since, by assumption, $\sder_1$ is the least such logic.
\end{proof}

We should be careful here, as the logic defined by a given Nmatrix 
 may well fail to be finitely axiomatizable (see, for instance,~\cite{wronski3val}). Further, even if the logic is finitely axiomatizable, obtaining a concrete axiomatization may be very difficult (see, for instance,~\cite{rautenberg}). Still, if the axiomatizations are given, or obtainable with some ingenuity (as is often the case), one can easily test the equivalence of the given Nmatrices by checking that each rule of one axiomatization holds for the other Nmatrix\footnote{Note that, in this case, it is not relevant whether the two logics were defined by Nmatrices, but simply that their corresponding consequence problems are decidable.}.

\begin{corollary}\label{axcrit}
Let $\Mt_1,\Mt_2$ be finite $\Sigma$-Nmatrices. Given finite axiomatizations of both $\tuple{\Sigma,\sder_{\Mt_1}},\tuple{\Sigma,\sder_{\Mt_2}}$ it is decidable whether $\Mt_1,\Mt_2$ are equivalent.
\end{corollary}

Let us provide some illustrations of this criterion at work. 

\begin{example}{(Eight distinct logics)}\label{eight}\\
Let $\Sigma$ be the signature with a single 1-place connective $\flat$ of Example~\ref{unarys}, and the Nmatrices introduced there. 
It is not difficult, for instance with the aide of techniques such as~\cite{SS,SYNTH,platypus}, to produce the following axiomatizations $R_\Mt$ for the logics defined by each 
$\Mt\in\{\Ut_\Sigma,\Mt_1,\ldots,\Mt_8\}\setminus\{\Mt_6\}$:
$R_{\Ut_\Sigma}=\emptyset$, $R_{{\Mt_1}}=\{r_{1}\}$, $R_{{\Mt_2}}=\{r_{2}\}$, 
$ R_{{\Mt_3}}=\{r_{3}\}$, $R_{{\Mt_4}}=\{r_4\}$, $R_{{\Mt_5}}=\{r_2,r_3\}$, $  R_{{\Mt_7}}=\{r_1,r_5,r_6\}$, $R_{{\Mt_8}}=\{r_7\}$, where the various rules are given below.
$$ \frac{\,p,\flat p\,}{q}\,_{r_{1}}\quad\quad 
\frac{p}{\,\flat p\,}\,_{r_{2}}\quad\quad
\frac{\,\flat p\,}{p}\,_{r_{3}}\quad\quad
\frac{\,\flat p\,}{q}\,_{r_{4}}\quad\quad
\frac{p}{\,\flat\flat p\,}\,_{r_{5}}\quad\quad
\frac{\,\flat\flat p\,}{ p}\,_{r_{6}}\quad\quad
  \frac{}{\,\flat p\,}\,_{r_{7}}
$$

Using Corollary~\ref{axcrit} it is easy to conclude that all these logics are distinct. Indeed, using Proposition~\ref{axprop} one can even assess their relationship more properly, as pictured in Figure~\ref{thepic} below.\smallskip

We have purposefully left Nmatrix $\Mt_6$ out of this analysis. We did it for two reasons: first, the techniques mentioned above do not readily provide us with a corresponding axiomatization; and second, it actually turns out that its logic coincides with one of the others. To prove this we will need to proceed differently.
\hfill$\triangle$
\end{example}

\subsection{Semantical approaches}

It is clear that the logic defined by a $\Sigma$-Nmatrix $\Mt$ is completely determined by its valuations $\Val(\Mt)$ and its set of designated values. Thus, ultimately, we can always try to assess whether given Nmatrices are equivalent by analyzing their valuations. Let us look at one example.

\begin{example}{(Another discrete logic)}\label{discrete}\\
Let $\Sigma$ be the signature with a single 1-place connective $\flat$ of Example~\ref{unarys}, and the Nmatrices introduced there.
We will show that $\Mt_6$ is actually equivalent to $\Ut_\Sigma$, both defining the discrete logic on the corresponding language. Since we know from Example~\ref{unconstrained} that $\sder_{\Ut_\Sigma}$ is discrete, it is immediate that 
$\sder_{\Ut_\Sigma}{\subseteq}\sder_{\Mt_6}$. To prove equality we need to show that given $\Gamma\cup\{A\}\subseteq L_\Sigma(P)$, if $\Gamma\sder_{\Mt_6} A$ then $A\in\Gamma$.

Fixed $A\in L_\Sigma(P)$, consider $v_A:L_\Sigma(P)\to\{0,1\}$ as defined below.
 $${v_A}(B)=\left\{\begin{array}{cl}
                       0 & \textrm{if $B=A$}\\
                       1 & \textrm{if $B\neq A$}
                       \end{array}\right.$$         
It is simple to check that $v_A\in\Val(\Mt_6)$. Just note that if $v_A(B)=0$ then $B=A$, and thus $v_A(\flat B)=v_A(\flat A)=1$ as $\flat A\neq A$. 
Therefore, if $A\notin\Gamma$ then $v_A(\Gamma)\subseteq\{1\}$ and $v_A(A)=0$, which shows that $\Gamma\not\sder_{\Mt_6} A$.
\hfill$\triangle$\smallskip
\end{example}

This example shows that semantic methods can be relatively \emph{ad hoc}, depending crucially on the concrete Nmatrices given and our ability to understand them. Still, there are some interesting structural criteria that can sometimes be used.

\subsubsection{Rexpansion homomorphisms}

Nmatrices can be related by an interesting notion of  
 homomorphism stemming from the notion of \emph{rexpansion}~\cite{rexpansions}.\smallskip

Let $\Mt_1 = \tuple{V_1, \cdot_1, D_1}$ and 
$\Mt_2 = \tuple{V_2, \cdot_2, D_2}$ be $\Sigma$-Nmatrices. A \emph{strict homomorphism} $h:\Mt_1\to\Mt_2$ consists of a function 
$h: V_1 \to V_2$ such that  $D_1 =h^{-1}(D_2)$ and,  
$h(\conn_{\Mt_1}(x_1, \ldots, x_k)) \subseteq \conn_{\Mt_2}(h(x_1), \ldots, h(x_k))$ for every $k\in\nats_0$, $\conn \in \Sigma^{(k)}$, and $x_1,\dots,x_k\in V_1$. 

Further, when $h$ is surjective, a strict homomorphism is said to be \emph{strongly-preserving} if 
$h(\conn_{\Mt_1}(x_1, \ldots, x_k))=\conn_{\Mt_2}(h(x_1), \ldots, h(x_k))$ for every  $k\in\nats_0$, $\conn \in \Sigma^{(k)}$, and $x_1,\dots,x_k\in V_1$.
We recall here Theorem 3.8 in~\cite{rexpansions}.

\begin{theorem} 
\label{thAZcompare}
If $h:\Mt_1\to\Mt_2$ is a strict homomorphism then
$\vdash_{\Mt_2}{\subseteq}\vdash_{\Mt_1}$. 
Further, if the homomorphism is strongly-preserving then
$\vdash_{\Mt_1}{=}\vdash_{\Mt_2}$.
\end{theorem}

Paper~\cite{rexpansions} contains several interesting applications of this result. We shall provide a few more, along also with an illustration of its limitations.

\begin{example} (Revisiting inclusions)\\
 We revisit the Nmatrices of Example~\ref{unarys}. Let $h:\{0,1\}\to \{0,1\}$ be the identity function. It is easy to see that $h$ defines strict homomorphisms
 $h:\Mt\to\Mt'$ for all pairs $\tuple{\Mt,\Mt'}$ in
 $$\{\tuple{\Mt_i,\Ut_\Sigma}:i\in\{1,\dots,8\}\}\cup\{\tuple{\Mt_i,\Mt_j}:(i,j)\in 
 \{(7,1),(8,2),(5,2),(5,3),(4,3)\}
 \},$$
 simply because 
   $\flat_{\Mt}(x)\subseteq \flat_{\Mt'}(x)$ for every $x\in \{0,1\}$.
   Thus, Theorem~\ref{thAZcompare} explains by itself all the inclusion arrows depicted in Figure~\ref{thepic} (which we had already justified in Example~\ref{eight}).\smallskip

Note, however, that $h:\Mt_6\to \Ut_\Sigma$ is not strongly preserving, and that $h$ does not even constitute a strict homomorphism from $\Ut_\Sigma$ to $\Mt_6$ (in both cases, simply because $0\notin \flat_{\Mt_6}(0)=\{1\}$), and therefore Theorem~\ref{thAZcompare} cannot be used to establish the equivalence of these two Nmatrices.
\hfill$\triangle$\smallskip
\end{example}

\begin{example} (Tilded inclusions and equivalences)\label{tildeight}\\
We extend Example~\ref{tilded}, by considering the tilded versions, according to Definition~\ref{def-infect}, of all the Nmatrices of Example~\ref{unarys}. Note that in all cases the tilded Nmatrices have both $1,\widetilde{0}$ as designated values. The corresponding truth-tables are given below.

\begin{center}
\begin{tabular}{C{7pt}|C{8mm}|C{8mm}|C{8mm}|C{8mm}|C{8mm}|C{8mm}|C{8mm}|C{8mm}|C{8mm}}
 & ${\flat_{\widetilde{\Ut}_\Sigma}}$ & ${\flat_{\widetilde{\Mt}_1}}$  & ${\flat_{\widetilde{\Mt}_2}}$  & ${\flat_{\widetilde{\Mt}_3}}$  & ${\flat_{\widetilde{\Mt}_4}}$  & ${\flat_{\widetilde{\Mt}_5}}$  & ${\flat_{\widetilde{\Mt}_6}}$  & ${\flat_{\widetilde{\Mt}_7}}$  & ${\flat_{\widetilde{\Mt}_8}}$ \\[1mm]
\hline
$0$& $ 0,\widetilde{0},1 $ & $ 0,\widetilde{0},1 $ & $ 0,\widetilde{0},1 $ & $ 0,\widetilde{0} $ & $ 0,\widetilde{0} $ & $ 0,\widetilde{0} $ & $ 1 $& $ 1 $& $ 1 $\\[1mm]
$\widetilde{0}$& $ 0,\widetilde{0},1 $ & $ 0,\widetilde{0},1 $ & $ 0,\widetilde{0},1 $ & $ 0,\widetilde{0} $ & $ 0,\widetilde{0} $ & $ 0,\widetilde{0} $ & $ 1 $& $ 1 $& $ 1 $\\[1mm]
$1$ &$ 0,\widetilde{0},1 $ &$ 0,\widetilde{0} $ &$ 1 $ &$ 0,\widetilde{0},1 $ &$ 0,\widetilde{0} $ &$ 1 $ &$ 0,\widetilde{0},1 $ &$ 0,\widetilde{0} $ &$ 1 $
\end{tabular}

\end{center}

Consider the functions $h:\{0, \widetilde{0}, 1\} \to \{0, 1\}$, and $f,g:\{0, 1\} \to \{0,\widetilde{0}, 1\}$ defined as follows.
$$h(x)=
\begin{cases}
 0 & \text{ if }x=0\\
 1 & \text{ if }x\in \{\tilde{0},1\}\\
\end{cases}
\qquad 
f(x)=
\begin{cases}
 0 & \text{ if }x=0\\
 \tilde{0} & \text{ if }x=1\\
\end{cases}
\qquad 
g(x)=
\begin{cases}
 0 & \text{ if }x=0\\
 1 & \text{ if }x=1\\
\end{cases}
$$

For any $\Mt \in \{\Ut_\Sigma, \Mt_1,\Mt_3,\Mt_4\}$, 
$h:\widetilde{\Mt}\to\Ut_\Sigma$ is a strongly-preserving strict homomorphism (because, for all $x\in\{0,\widetilde{0},1\}$, $\flat_{\widetilde{\Mt}}(x)$ contains both designated and undesignated values). 
Consequently, by Theorem~\ref{thAZcompare}, we conclude that 
$\vdash_{\widetilde{\Mt}} {=} \vdash_{\Ut_\Sigma}$.\smallskip

For any $\Mt \in \{\Ut_\Sigma, \Mt_1,\ldots,\Mt_5\}$, 
$f:\Ut_\Sigma\to\widetilde{\Mt}$ is a strict homomorphism (because, for all $x\in\{0,\widetilde{0}\}$, $\flat_{\widetilde{\Mt}}(x)$ contains both designated and undesignated values). 
By Theorem~\ref{thAZcompare}, we conclude that 
$\vdash_{\widetilde{\Mt}}{\subseteq}\vdash_{\Ut_\Sigma}$, and since 
$\vdash_{\Ut_\Sigma}$ is the smallest logic over $\Sigma$ they must coincide, that is, $\vdash_{\widetilde{\Mt}} {=} \vdash_{\Ut_\Sigma}$.\smallskip

For $\Mt \in \{\Ut_\Sigma, \Mt_1,\Mt_6, \Mt_7\}$, 
$g:\Mt_7\to\widetilde{\Mt}$ is a strict homomorphism (because $1\in\flat_{\widetilde{\Mt}}(0)$ and $0\in\flat_{\widetilde{\Mt}}(1)$).
Theorem~\ref{thAZcompare} thus implies that $\vdash_{\widetilde{\Mt}} {\subseteq} \vdash_{\Mt_7}$. 
Since $\vdash_{\Mt_7}$ corresponds to the negation fragment of classical logic, 
which has no theorems, it follows from Proposition~\ref{infect} that $\vdash_{\widetilde{\Mt}}$ also has no theorems, and $\vdash_{\widetilde{\Mt}} {=} \vdash_{\Ut_\Sigma}$.\smallskip

It is actually possible to conclude, simply by inspection of the axiomatizations provided in Example~\ref{eight}, or also semantically, that none of the logics defined by the given Nmatrices have theorems, with the exception of $\Mt_8$. Thus, according to  Proposition~\ref{infect}, their tilded versions must all be equivalent to $\Ut_\Sigma$.
We conclude that for all $\Mt \in \{\Ut_\Sigma, \Mt_1, \ldots, \Mt_7\}$, we have $\vdash_{\widetilde{\Mt}} {=} \vdash_{\Ut_\Sigma} {\neq} \vdash_{\widetilde{\Mt}_8}$, and also that 
if ${\Mt}\neq {\Ut_\Sigma}$ then 
$\vdash_{\Mt} {\neq} \vdash_{\widetilde{\Mt}}$.
It is important to note that Proposition~\ref{infect} was essential here, as Theorem~\ref{thAZcompare} could not be directly applied to show that $\vdash_{\widetilde{\Mt}_7} {=} \vdash_{\Ut_\Sigma}$ (because $0\in \flat_{\Ut_\Sigma}(0)$ but
 $0 \notin \flat_{\Mt_7}(0)$).\smallskip

Finally, $h:\widetilde{\Mt}_8\to\Mt_8$ is a strongly-preserving strict homomorphism. According to Theorem~\ref{thAZcompare}, we conclude that 
$\vdash_{\Mt_8} {=} \vdash_{\widetilde{\Mt}_8}$.\smallskip

The overall picture of the relationships between all these logics is depicted in Figure~\ref{thepic}.\hfill$\triangle$\smallskip
\end{example}

These examples are, of course, very simple. Still, the variety of possibilities comes as no surprise, since we have shown that ${\mathsf{Eqv}}(\mathsf{Nmatr})$ is undecidable, while checking the existence of suitable homomorphisms between finite Nmatrices is decidable.

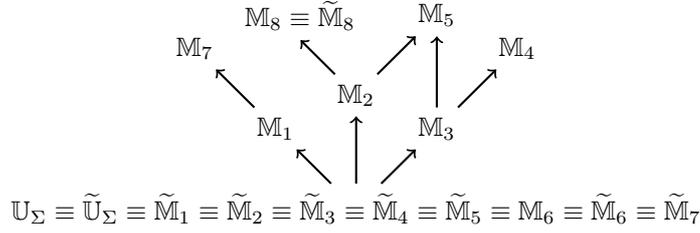
\begin{figure}[hb]
\begin{center}
\begin{tikzpicture}[node distance={15mm}, thick, main/.style = {}] 
\node[main] (1) {$\Ut_\Sigma\equiv\widetilde{\Ut}_\Sigma\equiv\widetilde{\Mt}_1\equiv\widetilde{\Mt}_2\equiv\widetilde{\Mt}_3\equiv\widetilde{\Mt}_4\equiv\widetilde{\Mt}_5\equiv\Mt_6\equiv\widetilde{\Mt}_6\equiv\widetilde{\Mt}_7$}; 
\node[main] (2) [above left of=1] {$\Mt_1$}; 
\node[main] (3) [above of=1] {$\Mt_2$}; 
\node[main] (4) [above right of=1] {$\Mt_3$}; 
\node[main] (5) [above left of=2] {$\Mt_7$}; 
\node[main] (6) [above left of=3] {$\;\;\;\quad \Mt_8\equiv\widetilde{\Mt}_8$}; 
\node[main] (7) [above right of=3] {$\Mt_5$}; 
\node[main] (8) [above right of=4] {$\Mt_4$}; 
\draw[->] (1) -- (2); 
\draw[->] (1) -- (3); 
\draw[->] (1) -- (4); 
\draw[->] (2) -- (5);
\draw[->] (3) -- (6); 
\draw[->] (3) -- (7); 
\draw[->] (4) -- (7); 
\draw[->] (4) -- (8); 
\end{tikzpicture} 
\end{center}
\caption{Equivalences ($\equiv$) and inclusions ($\longrightarrow$) among the eight logics defined by the eighteen Nmatrices of Examples~\ref{unarys} and~\ref{tildeight}.}\label{thepic}
\end{figure}

\subsubsection{ Bivaluations }
 
 We conclude with a very general semantical criterion, rooted in bivaluations, that is necessary and sufficient to prove or disprove the equality of two logics. Although this criterion does not yield a direct procedure, it provides a faithful conceptual framework for understanding the problem.\smallskip

 Fixed a signature $\Sigma$, bivaluations
are functions $\bval:L_{\Sigma}(P)\to\{0,1\}$.
Any Nmatrix induces a set of bivaluations
 $\BVal(\Mt)=\{\bval_v:v\in \Val(\Mt)\}$ where, for each formula $A$,
$$b_v(A)=
\begin{cases}
 1& \mbox{ if }v(A)\in D\\
0 &\mbox{ if } v(A)\notin D\\
\end{cases}.
$$

    Given $\bvals\subseteq\{0,1\}^{L_{\Sigma}(P)}$, let its \emph{meet-closure} be the set 
$\bvals^\sqcap=\{\sqcap{\bvals'}:\bvals'\subseteq \bvals\}$ with each \emph{meet} $\sqcap{\bvals'}$ being the bivaluation such that $\sqcap{\bvals'}(A)=1$ precisely if $b(A)=1$ for every $b\in \bvals'$, for each $A\in L_{\Sigma}(P)$.

The following is a well-known simple result about theories of a logic, rephrased here as bivaluations (given $b\in\BVal(\Mt)$, it is straightforward to check that $\Omega_b=b^{-1}(1)=\{A\in L_\Sigma(P):b(A)=1\}$ is a theory of $\sder_\Mt$, i.e., if $\Omega_b\sder_\Mt A$ then $A\in\Omega_b$).

\begin{proposition} 
 Given $\Sigma$-Nmatrices $\Mt_1$ and $\Mt_2$,
  we have that
  $$\sder_{\Mt_1} {\subseteq} \sder_{\Mt_2} \textrm{ {if and only if} } \BVal(\Mt_2)\subseteq\BVal(\Mt_1)^\sqcap.$$
\end{proposition}
\begin{proof}
Let $\Mt_1 = \tuple{V_1, \cdot_1, D_1}$ and 
$\Mt_2 = \tuple{V_2, \cdot_2, D_2}$.\smallskip

Assume that $\sder_{\Mt_1} {\subseteq} \sder_{\Mt_2}$, and take $b\in \BVal(\Mt_2)$. If $\Omega_b\sder_{\Mt_1} A$ then $\Omega_b\sder_{\Mt_2} A$, and thus $A\in\Omega_b$. Thus, for each $A\notin\Omega_b$ it must be the case that $\Omega_b\not\sder_{\Mt_1} A$, and thus there exists $v_A\in\Val(\Mt_1)$ such that $v_A(\Omega_b)\subseteq D_1$ and $v_A(A)\notin D_1$. Denoting each $b_{v_A}$ simply by $b_A$, note that $b_A(\Omega_b)\subseteq\{1\}$ and $b_A(A)=0$. We just need to check that $b=\sqcap\{b_A:A\notin\Omega_b\}$. Indeed, for any formula $B$, $b(B)=1$ iff $B\in\Omega_b$ iff $b_A(B)=1$ for all $A\in\Omega_b$.\smallskip

Reciprocally, assume that $\BVal(\Mt_2)\subseteq\BVal(\Mt_1)^\sqcap$, take $\Gamma\cup\{A\}\subseteq L_\Sigma(P)$ such that $\Gamma\sder_{\Mt_1} A$, and let $v_2\in\Val(\Mt_2)$ be such that $v_2(\Gamma)\subseteq D_2$. Clearly, $b_{v_2}(\Gamma)\subseteq\{1\}$ and $b_{v_2}=\sqcap \bvals$ for some set $\bvals\subseteq\BVal(\Mt_1)$. This means that there exists ${\mathsf V}\subseteq\Val(\Mt_1)$ such that $\bvals=\{b_{v_1}:v_1\in {\mathsf V}\}$, and for any formula $B$ we have that $v_2(B)\in D_2$ iff $b_{v_2}(B)=1$ iff 
$b_{v_1}(B)=1$ for all $v_1\in {\mathsf V}$ iff $v_1(B)\in D_1$ for all $v_1\in {\mathsf V}$. Since $v_2(\Gamma)\subseteq D_2$, it follows that $v_1(\Gamma)\subseteq D_1$ for all $v_1\in {\mathsf V}$. Therefore, given that $\Gamma\sder_{\Mt_1} A$ we conclude that $v_1(A)\in D_1$ for all $v_1\in {\mathsf V}$. Consequently $v_2(A)\in D_2$, and it follows that $\Gamma\sder_{\Mt_2} A$.
\end{proof}

\begin{corollary}\label{bivs}
Given $\Sigma$-Nmatrices $\Mt_1$ and $\Mt_2$,
  we have that
  $$\sder_{\Mt_1} {=} \sder_{\Mt_2} \textrm{ {if and only if} } \BVal(\Mt_1)^\sqcap=\BVal(\Mt_2)^\sqcap.$$
\end{corollary}
 
 This characterization gives us a necessary and sufficient criterion, which can be used as a ultimate resource.
 
\begin{example} ($\Mt_6$ revisited)\\
Let us revisit Example~\ref{discrete}. An alternative proof of the fact that $\Mt_6$ and $\Ut_\Sigma$ are equivalent would be to use Corollary~\ref{bivs}. Since both Nmatrices are two-valued, their valuations coincide with their bivaluations. Clearly, $\Val(\Ut_\Sigma) =  \Val(\Ut_\Sigma)^\sqcap = \{0,1\}^{L_{\Sigma}(P)}$. Therefore, it suffices to show that $\{0,1\}^{L_{\Sigma}(P)} \subseteq  \Val(\Mt_6)^\sqcap$.

Let $v:L_\Sigma(P)\to\{0,1\}$ be an arbitrary (bi)valuation. Consider the set of (bi)valuations
$\bvals = \{v_A: A\in v^{-1}(0)\}$, where each $v_A\in\Val(\Mt_6)$ is defined as in Example~\ref{discrete} and satisfies, for any formula $B$, $v_A(B)=1$ iff $B\neq A$. 
Note that $v(B)=1$ iff $B\notin v^{-1}(0)$ iff $v_A(B)=1$ for all $v_A\in \bvals$, i.e., $v=\sqcap \bvals$.
 \hfill$\triangle$\smallskip
\end{example}

\section{Conclusion}\label{sec6:conc}

We have shown that the price to pay for the expressive gain allowed by non-deterministic logical matrices includes the loss of the ability to decide logical equivalence. 
This undecidability result implies that looking for counterexamples to the logical equivalence of two given finite Nmatrices can be very difficult in practice, let alone proving their logical equivalence, if that is the case. \smallskip

These results do not hinder the usefulness of Nmatrices, but do require additional expertise in dealing with their equivalence problem. We have briefly surveyed some possible pathways to deal with these problems, but further research is needed. Namely, it will be very useful to identify subclasses $\textsf{Matr}\subseteq\mathcal{M}\subseteq\textsf{NMatr}$ for which the problem $\mathsf{Eqv}(\textsf{NMatr})$ is still decidable. Also, alternative sufficient conditions for logical equivalence may result from a deeper study of the algebraic structure of Nmatrices and their homomorphisms, namely in the lines of~\cite{CZE}. Further, the scope of the idea underlying the equivalence criterion of Corollary~\ref{axcrit} and used above in Example~\ref{eight} seems to be extensible to richer notions of axiomatizability. 

For instance, we know that every finite logical matrix~\cite{SS}, and also all every \emph{monadic} finite Nmatrix~\cite{SYNTH} admits a finite multiple-conclusion axiomatization. We have shown in~\cite{FilipeNCL} that determining whether a given finite Nmatrix is monadic is undecidable, but given formulas witnessing its monadicity we know how to algorithmically produce a corresponding multiple-conclusion axiomatization.
Furthermore, under certain conditions, finite multiple-conclusion axiomatizations can be converted into finite single-conclusion axiomatizations. For instance, \cite[Theorem~5.37]{SS} shows that this is possible when the underlying logic expresses a \emph{disjunction-like} connective. %
Even if there is no disjunction, we have that by  \cite[Theorem~5.16]{SS}
we have that if
 $R=\{\frac{\,\Gamma_i\,}{\Delta_i}:i\in I\}$ is a multiple-conclusion axiomatization of $\tuple{\Sigma,\sder_1}$ 
 then $\sder_1{\subseteq} \sder_2$ if, and only if,
for every $\Theta\cup \{A\} \subseteq L_{\Sigma}(P)$, $i\in I$ and $\sigma:P\to L_\Sigma(P)$ we have that
\begin{center}
 $\Theta,B^\sigma\sder_1 A$ for every $B\in \Delta_i$ 
implies
 $\Theta,\Gamma_i^\sigma\sder_2 A$.
 \end{center}

Similar ideas may be workable, using with other notions of axiomatizability, for instance using sequent or tableaux calculi, as long as there is a way to convert them into tests to the Nmatrix logics. \smallskip

Finally, it is worth noting that if we adopt a non-Tarskian multiple-conclusion notion of logic, as proposed in~\cite{Sco:CaA:74,SS}, the work developed here is inconclusive. Indeed, the problem of determining whether two given finite Nmatrices characterize the same multiple-conclusion logic (which does not seem simpler) remains open, as one can easily note that the Nmatrix $\Ct$ obtained from a given counter machine $\mC$ in Section~\ref{sec3} is not saturated (in the sense of~\cite{newfibring}), and it is unclear how to replace the construction in a suitable way.

\bibliographystyle{plain}
\bibliography{biblio}

\end{document}